\documentclass{article}
\usepackage{hyperref}
\usepackage{tocbibind}
\usepackage[T1]{fontenc}
\usepackage[utf8]{inputenc}
\usepackage{enumerate}
\usepackage{changes}
\usepackage{hyperref}
\usepackage{amsmath}
\usepackage{amsfonts}
\usepackage{amssymb}
\usepackage{amsthm}
\usepackage{newlfont}
\usepackage{mathtools}
\usepackage[top=1.2in, bottom=1.3in, left=.9in, right=.9in]{geometry}

\usepackage{authblk}
\usepackage{xcolor}
\usepackage{thmtools}

\theoremstyle{plain}
\newtheorem{theorem}{Theorem}[section]
\newtheorem{proposition}[theorem]{Proposition}
\newtheorem{corollary}[theorem]{Corollary}
\newtheorem{lemma}[theorem]{Lemma}

\newtheorem{assumption}[theorem]{Assumption}

\theoremstyle{definition}
\newtheorem{definition}[theorem]{Definition}

\newtheorem*{remark}{Remark}

\newcommand{\RR}{\mathbb{R}}
\newcommand{\NN}{\mathbb{N}}
\newcommand{\ZZ}{\mathbb{Z}}

\newcommand{\cC}{\mathcal{C}}
\newcommand{\drm}{\mathrm{d}}
\newcommand{\euler}{\mathrm{e}}

\DeclareMathOperator{\supp}{supp}

\DeclareMathOperator{\Deg}{Deg}

\DeclareMathOperator{\spec}{spec}
\DeclareMathOperator{\av}{av}
\DeclareMathOperator{\arsinh}{arsinh}

\DeclareMathOperator{\Eins}{\mathbf{1}}
\DeclareMathOperator{\diam}{diam}

\newcommand{\eat}[1]{}

\let\oldint\int
\renewcommand{\int}{\oldint\limits}

\newcommand{\Hmm}[1]{\leavevmode{\marginpar{\tiny%
			$\hbox to 0mm{\hspace*{-0.5mm}$\leftarrow$\hss}%
			\vcenter{\vrule depth 0.1mm height 0.1mm width \the\marginparwidth}%
			\hbox to
			0mm{\hss$\rightarrow$\hspace*{-0.5mm}}$\\\relax\raggedright #1}}}

\begin{document}

\title{Gaussian upper heat kernel bounds and Faber-Krahn inequalities on graphs
}
\author{Christian Rose\thanks{christian.rose@uni-potsdam.de}\thanks{Support by the DFG grant  "\emph{Heat kernel behavior at infinity on graphs}", project number 540199605,
is gratefully acknowledged. }
}
\affil[]{Institut f\"ur Mathematik, Universit\"at Potsdam, 		14476  Potsdam, Germany}
\date{\today}
\maketitle
\begin{abstract}
We investigate the equivalence of relative Faber-Krahn inequalities and the conjunction of Gaussian upper heat kernel bounds and volume doubling on large scales on  graphs.  
For the normalizing measure, we obtain their equivalence up to constants. 
If the counting measure or arbitrary measures are considered, a local regularity condition containing the vertex degree and a variable dimension enter the equivalence. Correction functions for the Gaussian, doubling, and Faber-Krahn dimension depending on the vertex degree are introduced. For the Gaussian and doubling, these correction functions always tend to one at infinity. Moreover, the variable Faber-Krahn dimension  can be related to the doubling dimension and the vertex degree growth. 
\\

\end{abstract}
\tableofcontents
\section{Introduction and main results}
This article is concerned with Gaussian upper bounds of continuous time heat kernels on discrete weigthed graphs with possibly unbounded geometry. In this general setting,
we aim at a characterization of the conjunction of Gaussian upper heat kernel bounds and the volume doubling property and relative Faber-Krahn inequalities.  
On Riemannian manifolds, this equivalence was obtained by Grigor'yan \cite{Grigoryan-94}, cf.~\cite{Grigoryan-09}. Discrete-time heat kernels on graphs with normalizing measure were treated  in \cite{CoulhonG-98}. For characterizations of Gaussian bounds in different contexts, we refer to \cite{Varopoulos-85,CarlenKS-87,Coulhon-92,SaloffCoste-92,Grigoryan-94,
SaloffCoste-92a,Delmotte-99,SaloffC-01,BCS, BarlowChen-16, GrigoryanHH} for a non-exhaustive list of references.

The interest in such equivalences stems from the fact that Faber-Krahn inequalities are easily controlled by geometric terms like isoperimetric constants via Cheeger's inequality, \cite{BauerKW-15}. Moreover, there are settings in which Faber-Krahn inequalities have certain stability properties, e.g., under rough isometries \cite{Barlow-book}, which is not evident for Gaussian upper bounds for the heat kernel.

Pang and Davies observed that the continuous-time heat kernel on $\ZZ$ with standard weights and normalizing measure exhibits Gaussian behavior for large but not small times \cite{Davies-93, Pang-93}. Moreover, the well-known Varadhan asymptotics for small times on Riemannian manifolds fails on all graphs, \cite{KellerLVW-15}. Hence, characterizations of Gaussian upper bounds for continuous-time heat kernels in the spirit of Grigor'yan's result can only be expected if arbitrarily small times are excluded. This was treated for the normalizing measure, e.g., in \cite{Barlow-book}.

The ideas needed to obtain heat kernel upper bounds have proved very fruitful in the study of random walks in symmetric random environments. A particular example was given by Barlow, who obtained Gaussian bounds on the heat kernel with normalizing measure on supercritical percolation clusters in $\ZZ^d$, \cite{Barlow-04}.
A general feature of random environments on graphs is that conditions required for heat kernel bounds only hold for sufficiently large balls.

Laplacians on graphs define Dirichlet forms of pure jump type and heat kernels without (sub-) Gaussian behavior for small times. By definition, the vertex measure on a discrete graph is purely atomic. The discreteness of the space leads to a violation of the generalized capacity condition and its equivalent properties as well as the desired upper heat kernel bounds for jump processes considered in, e.g., 
\cite{GrigoryanHH}.
One key to resolve these problems is the employment of intrinsic metrics.
The discovery of intrinsic metrics for strongly local and regular Dirichlet forms in \cite{Sturm-94,FrankLenzWingert-14} led to systematic treatement of such metrics on graphs as well. 
After their first explicit appearance for graphs in \cite{Folz-11,GrigoryanHuangMasamune-12}, they turn out to be vital in the study of graphs with unbounded geometry, \cite{BauerKH-13, HaeselerKW-13, Folz-14b, Folz-14, BauerKW-15,Keller-15, HuangKS-20, KellerLW-21}.

Davies observed that heat kernel bounds for the counting measure improve by choosing metrics different from the combinatorial one, which is typically related to the normalizing measure, \cite{Davies-93a}. Gaussian upper heat kernel bounds for elliptic operators with integrable weights on $\ZZ^n$ with respect to the counting measure and the chemical distance have been provided by \cite{AndresDS-16,Bella-22}. Recently, Keller and the author characterized continuous-time heat kernel upper bounds in terms of Sobolev inequalities on large balls, \cite{KellerRose-22a, KellerRose-24}. That the latter follow from isoperimetric inequalities and in turn holds for anti-trees was discovered in \cite{KellerRose-22b}.

In this paper, we extend Grigor'yan's above mentioned equivalence to deterministic graphs such that, for all large enough balls, relative Faber-Krahn inequalities hold with respect to intrinsic metrics. Contrary to earlier works on this equivalence, we obtain Davies' and Pang's sharp Gaussian behavior of the heat kernel in terms of such metrics. Our work delivers more precise sufficient conditions on the range of balls for which Faber-Krahn inequalities are required in order to obtain heat kernel upper bounds.
One main result for the normalizing measure is the analogue of Grigoryan's theorem up to a uniform regularity condition on the smallest considered ball. A further generalization of this result allows to remove this assumption on the cost of a new local regularity condition becoming part of the equivalence and a variable dimension in the Faber-Krahn inequality. If the counting measure or general measures are considered, this dimension relates to the volume doubling dimension and the  vertex degree growth. Moreover, the Gaussian bounds and the volume doubling property contain terms involving the vertex degree and become neglible on large scales.

\subsection{The set-up}
Let $X$ be an at most countable set and
extend a function
 $m\colon X\to(0,\infty)$ of full support to a measure on $X$ via countable additivity.
A symmetric function $b\colon X\times X\to [0,\infty)$  with 
\[
b(x,x)=0\quad\text{and}\quad \deg_x:=\sum_{y\in X}b(x,y)<\infty, \quad x\in X,
\]
is called a \emph{graph} over the measure space $(X,m)$. {We write $ x\sim y $ whenever $ b(x,y)>0 $ for $ x,y\in X $.} Furthermore, denote the (weighted) vertex degree by 
\[
\Deg_x:=\frac{\deg_x}{m(x)}=\frac{1}{m(x)}\sum_{y\in X}b(x,y),\qquad x\in X.
\]

Throughout this article, we assume that graphs are \emph{locally finite}, i.e., the set $\{y\in X\colon b(x,y)>0\}$ is finite for any $x\in X$. 

The Laplace operator $\Delta\colon \cC(X)\to\cC(X)$ on a graph acts on the set of functions $\cC(X)$ as
\[
\Delta f(x)=\frac1{m(x)}\sum_{y\in X} b(x,y)(f(x)-f(y)), 
\] 
for $ f\in \cC(X) $ and $  x\in X  $.
By local finiteness, $ \Delta $ maps the set of compactly supported functions $ \cC_{c}(X) $ into itself. Hence, the restriction of $ \Delta $ to $ \cC_c(X) $ is symmetric in $ \ell^{2}(X,m)$, the set of square integrable functions on $(X,m)$.  By slight abuse of notation we also denote the closure of $\Delta\restriction_{\cC_c(X)}$ in $\ell^2(X,m)$ by   $\Delta\geq 0$.
Note that $\Delta$ is bounded if and only if $\Deg$ is uniformly bounded in $X$. 

The heat semigroup $(\euler^{-t\Delta})_{t\geq 0}$, acts in $ \ell^{2}(X,m) $ and has a kernel $p\colon [0,\infty)\times X\times X\to[0,\infty)$, called the \emph{heat kernel},
satisfying
\[
\euler^{-t\Delta}f(x)=\sum_{x\in X} m(x) p_t(x,y)f(y)
\]
for all $ f\in \ell^2(X,m),  x\in X, t\geq 0 $.
For an initial condition $ f\in \ell^{2}(X,m) $, the function $ u=\euler^{-t\Delta}f $ solves the heat equation
\[
\frac{d}{dt}u=-\Delta u, \quad u(0,\cdot)=f.
\] 

The heat kernel of a graph corresponds to the transition density of the jump process on the underlying graph associated to the Laplacian $\Delta$. Two 
 examples for the measure $m$ are frequently discussed in the literature. Laplacians with normalizing measure $m=\deg$ generate so-called constant speed random walks. Variable speed random walks on graphs correspond to the counting measure $m=1$.
\\

In order to formulate our results we recall the definition of intrinsic metrics and list the required assumptions. 
A pseudo-metric $\rho\colon X\times X\to[0,\infty)$ on a graph $b$ on $(X,m)$ is called {intrinsic} if 
\[
\sum_{y\in X} b(x,y)\rho^2(x,y)\leq m(x),
\]
for all  $ x\in X $. We always abbreviate $\diam X:=\diam_\rho(X)=\sup\{\rho(x,y)\colon x,y\in X\}$ if there is no confusion about the intrinsic metric $\rho$. The {jump size} of an intrinsic metric $\rho$ is given by 
\begin{align*}
	S=\sup\{\rho(x,y)\colon b(x,y)>0, x,y\in X \}.
\end{align*}
Note that in the case $m=\deg$, the combinatorial distance $\rho_c$ is automatically intrinsic with jump size one.
As usual, for $r\geq 0 $, $x\in X$, we denote distance balls with respect to $\rho$ by $ B_x(r):=B_x^\rho(r)=\{y\in X\mid \rho(x,y)\leq r\}$.
A pseudo metric $ \rho  $ on a graph $b$ over $(X,m)$ is called  {path metric}  if there is $ w: X\times X\to [0,\infty] $ such that for all $ x,y\in X $
\begin{align*}
	\rho (x,y)=\inf_{x=x_{0}\sim \ldots\sim x_{n}=y}\sum_{j=1}^{n}w(x_{j-1},x_{j})
\end{align*}
{and $w(x,y)<\infty $ if and only if $ x\sim y $.}
Note that the choice $ w(x,y)=(\Deg_{x}\vee\Deg_{y})^{-1/2} $ for $ x\sim y $ and $ w(x,y) =\infty$ otherwise always yields an intrinsic path metric.

\begin{assumption}\label{assumption2} We assume that the intrinsic metric is a path metric whose distance balls are compact and that the jump size satisfies $S<\infty$.
\end{assumption}

Assumption~\ref{assumption2} yields that the metric space $ (X,\rho) $ is complete and geodesic, i.e., any two vertices $ x,y\in X $ can be connected by a path $x= x_{0}\sim\ldots \sim x_{n}=y $ such that $ \rho(x,y)=\rho(x,x_{j})+\rho(x_{j},y) $ for all $ j=0,\ldots,n $, see \cite[Chapter~11.2]{KellerLW-21}.
\\

In the following, we denote by $\|\cdot\|$ the norm in $\ell^2(X,m)$. For $f\in \cC(X)$ and $x,y\in X$ we abbreviate $\nabla_{xy}f:=f(x)-f(y)$ and set
\[
|\nabla f|^2(x):=\frac{1}{m(x)}\sum_{y\in X}b(x,y)(\nabla_{xy}f)^2.
\]
Recall the first Dirichlet eigenvalue of subsets $U\subset X$ 
\[
\lambda (U)=\inf_{f\in\cC(U),f\neq 0}\frac{\sum_U f\Delta f}{\Vert f\Vert^2}=\inf_{f\in\cC(U),f\neq 0}\frac{1}{2}\frac{\||\nabla f|\|^2}{\|f\|^2},
\]
and the bottom of the spectrum of the Laplacian
\[
\Lambda=\inf\spec(\Delta)=\inf_{U\subset X}\lambda(U).
\]

We will be dealing with the following properties a graph might or might not satisfy.

\begin{definition}Let $0\leq R_1\leq R_2$ and functions $a\colon X\times [R_1,R_2]\to(0,1]$, $n,\nu \colon X\times [R_1,R_2]\to(0,\infty)$, $\Psi\colon X\times X\times [R_1,R_2]\to (0,\infty)$, $\Phi\colon X\times [R_1,R_2]\times [R_1,R_2]\to (0,\infty)$, and  $B\subset X$.
\begin{itemize}
\item[(FK)] The \emph{relative Faber-Krahn inequality} $FK(R_1,R_2,a,n)$ is satisfied in $B$, if for all $x\in B$, $r\in[R_1,R_2]$, and $U\subset B_x(r)$ we have
\[
\lambda(U)\geq \frac{a_x(r)}{r^2}\left(\frac{m(B_x(r))}{m(U)}\right)^{\frac{2}{n_x(r)}}.
\]
We abbreviate $FK(R_1,a,n):=FK(R_1,\infty,a,n)$.
\item[(G)] \emph{Gaussian upper bounds} $G(R_1,R_2,\Psi,n)$ hold in $B$ if for all $x,y\in B$ and $t\geq R_1^2$ the heat kernel satisfies
\[
p_t(x,y)\leq \Psi_{xy}(\sqrt t\wedge R_2)
\frac{ 
\left(1\vee \frac{t}{S^2}\arsinh^2\left(\frac{\rho(x,y) S}{t}\right)\right)^{\frac{n_{xy}(\sqrt t\wedge R_2)}{2}}
}{\sqrt{m(B_x(\sqrt t\wedge R_2))m(B_y(\sqrt t\wedge R_2))}}
\euler^{-\Lambda(t-t\wedge R_2^2)-\zeta(\rho(x,y),t)}
\]
where $n_{xy}(\sqrt t\wedge R_2):= \tfrac{n_x(\sqrt t\wedge R_2)+n_y(\sqrt t\wedge R_2)}2$ and 
\[
\zeta\big(\rho,{t}\big)
=
\frac{\rho}{S}\arsinh\left(\frac{\rho S}{{t}}\right) 
-\frac{1}{S^2}\left(\sqrt{\rho^2S^2+{t}^2} -{t}
\right),\qquad \rho\geq 0, t>0.\]
We abbreviate $G(R_1,\Psi,n):=G(R_1,\infty,\Psi,n)$.
\item[(VD)]The \emph{volume doubling property} $V(R_1,R_2,\Phi,n) $ is satisfied in $B$ if for all $x\in B$ and $R_1\leq r\leq R\leq R_2$, we have 
\[
\frac{m(B_x(R))}{m(B_x(r))}\leq \Phi_x^{n_x(R)}(r)\left(\frac{R}{r}\right)^{n_x(R)}.
\]
Furthermore, property $V(R_2,\Phi,n)$ holds if the above inequality is true for all $x\in B$ and $0<r\leq R\leq R_2$. We abbreviate $V(\Phi,n):=V(\infty,\Phi,n)$.
\end{itemize}
\end{definition}

\begin{remark}
 If the functions $n,\phi,\Psi,\mu$ in the definitions above are constants, we will mention their constancy explicitly.
\end{remark}

\begin{remark}By the work of Davies and Pang, the function $\zeta$ in the case $S=1$ is sharp for the normalizing measure on the integers, 
\cite{Davies-93, Pang-93}. Moreover, we have for $r>0$
\[
\zeta(r,t)\ {\sim} \ \frac{r^2}{2t},\qquad t\to \infty,
\]
where  {$\sim $} 
means that the left-hand side divided by the right-hand side converges to one. 
\end{remark}

\subsection{Global variants of the main results}

In this section we present global versions of the characterizations announced in the introduction. We first restrict to the most common choices of normalizing and counting measure and elaborate on the differences before we present a version which is valid for all weighted graphs. All our results are obtained from localized versions which can be found in the subsequent sections with precise ranges of parameters. 
\\

First, we present our results about the normalizing measure $m=\deg$. In this case, the Laplacian $ \Delta $ is bounded on $ \ell^{2}(X,\deg) $.

\begin{theorem}[normalizing measure, uniform version]\label{thm:normalizedmain}
Let $m=\deg$, $\diam X=\infty$, and $n>0$, $R\geq 300$ constants.
If there exists a constant $a>0$ such that $FK(R,a,n)$ holds in $X$, then there exists a constant $C=C_{a,n,R_1}\geq 1$ such that $G(R,C,n) $ and  $V(C,n) $ hold in $X$.
Conversely, if there exists $C>0$ such that $G(R,C,n) $ 
and $ V(C, n)$ hold in $X$, 
then there exists $a'=a'_{n,C,R_1}>0$ such that $FK(R,a',n)$ holds in $X$.
\end{theorem}

\begin{remark}Theorem~\ref{thm:normalizedmain} is the global version of the more general and localized results presented in Corollary~\ref{cor:asympdoublingnormalized}, Corollary~\ref{cor:gaussianboundcleannormalized} and Theorem~\ref{thm:normalizedclassicalFK}.
\end{remark}

Next, we present our results for the special case of the counting measure $m=1$ and an intrinsic metric of jump size $S$. In order to obtain a characterization for the heat kernel bounds analogous to the case of normalizing measure, we have to incorporate the following local regularlity into the characterization. 

\begin{definition}Let $0\leq R_1\leq R_2$ functions $n,\nu \colon X\times [R_1,R_2]\to(0,\infty)$, and  $B\subset X$.
\begin{itemize}
\item[(L)] The local regularity property $L(R_1,R_2,\nu,n) $ is satisfied in $B$, if for all $x\in B$ and $r\in[R_1,R_2]$ we have 
\[
\frac{m(B_x(r))}{m(x)}
\leq \nu_x(r)\ r^{n_x(r)}.
\]
We abbreviate $L(R_1,\nu):=L(R_1,\infty,\nu)$.
\end{itemize}
\end{definition}

\begin{remark}The property (L) can be interpreted as volume doubling from balls with large radius to balls with vanishing radius. It however fails to be equivalent for small radii $ r $. The first appearance of such a uniform assumption can be traced back to the much stronger local regularity property used in \cite{BarlowChen-16}. 
\end{remark}

The incorporation of (L) and the possible unboundedness of the local geometry result in correction terms depending on the vertex degree. As mentioned in the introduction, the Gaussian bounds, and doubling property will be decorated with terms involving the vertex degree which vanish on large scales. In contrast, the local regularity property depends pointwise on a constant vertex degree. Further, the variable dimension function appearing in the Faber-Krahn inequality depends on the growth rate of the vertex degree. 

To phrase this behavior, we will use in the following the convention that if $f$ is a real map on $W$ being of the form $B$ or $B\times [a,b]$ for $B\in X$ and $a,b\in\RR$, then 
\[
\|f\|_W:=\sup_{W}|f|.
\]

For the next result, we introduce the following quantities.
Let $S,n>0$, $R_1,\tau,r\geq 0$, $x,y\in X$,
\[
\nu_x=\nu_{x,n}=[1\vee\deg_x]^{\frac{n}{2}},\qquad 
\Phi_x(r)=\Phi_x^n(r)=
\begin{cases}
\nu_{x}^{\theta(R_1)}R_1^n\nu_x&\colon r<R_1,\\[1.5ex]
\nu_{x}^{\theta(r)}&\colon \text{else},
\end{cases}
\quad
\Psi_{xy}(\sqrt \tau)^2=\Phi_x(\sqrt{\tau_\rho})\Phi_y(\sqrt{\tau_\rho}),
\]
where 
\[
\theta(r)=\theta(n,r)=C_\theta\cdot \left(\frac{S}{r}\right)^{\frac{1}{n+2}}, \qquad C_\theta=(288)^{\frac{1}{n+2}}\frac{n+2}{n},
\qquad
\tau_\rho:=\tau_{\rho(x,y)}=\frac{\tau}{2}\wedge \frac{S^2}{2\arsinh^2(\rho(x,y) S/\tau)}.
\]
Further, for
$C>0$, $R'= \euler^{1\vee 8R_1}$, $o\in X$, let a dimension function $n'\colon X\times [R',\infty)\to (0,\infty)$ be given by
\[
n_o'(r)=
n\cdot 1\vee \ln \left[(A(r)r)^2(1\vee \|\deg\|_{B_o(A(r)r)})\right]^{\iota(r)+\vartheta(r)},
\]
where 
\[A(r)=\exp\left(2^{19n}\euler C\cdot 1\vee\|\deg\|_{B_o(r)}^{\vartheta(r)}\right),\]
and
\[\iota(r)=\iota(n,r)=\frac{1}{2\ln \frac{r}{(\ln r)^{n+3}}},\qquad\vartheta(r)=\vartheta(n,r)=\frac{C_{n,S}}{(\ln r)^{\frac{n+3}{n+2}}}, \qquad C_{n,S}=176C_\theta S^{\frac{1}{n+2}}.\]
Finally, denote
\[ a'(r)=\big({2^{36}C^{26/n}(\ln A(r))^{2/n}}{A(r)^2}\big)^{-1}.
\]
The following  is our main result for the counting measure.
\begin{theorem}[counting measure]\label{thm:countingmainglobal}Let $m=1$. There exists $R_0=R_0(S)>0$ such that for all $R\geq R_0$
we have the following.
\begin{enumerate}[(i)]
\item If there are constants $a,n>0$ such that $FK(R,a,n)$ holds in $X$, then there is a constant $C=C_{a,n,S}\geq1$ such that $G(R,C\Psi,n)$, $L(R,C\nu,n) $, and $V(C\Phi,n)$ hold in $X$.
\item If there are constants $C,n>0$ such that $G(R,C\Psi,n)$, $L(R,C\nu,n) $, and  $V(C\Phi,n)$ hold in $X$, then we have $FK(8R',a',n')$ in $X$.
\end{enumerate}
In particular, $R_0(S)\to R_0(0)>0$ and $C_{a,n,S}\to C_{a,n,0}<\infty$ for $S\to 0$. 
\end{theorem}

\begin{remark}[behavior of the correction terms and dimension function] Observe that since $\theta$ is monotone decreasing, $\Phi_{x}(r) $ and hence  $ \Psi_x(r) $  are always bounded for large $r$ in terms $\deg_x$ and converge to one as $r\to \infty$. 
\\
 	Classically on manifolds, the dimension $ n\geq1 $ in the Faber-Krahn inequality which is derived from Gaussian upper estimates and volume doubling coincides with the volume doubling dimension. Here, the Faber-Krahn dimension is a function $ n' $.
 	The dimension $ n $ can be recovered asymptotically.
In case of polynomially bounded degree, the Sobolev dimension is asymptotically increased in terms of the polynomial growth rate of the degree.
 \\
Assume  $\|\deg\|_{B_o(r)}\in O(r^k)$ as $r\to\infty$
 for some $ k\geq 0 $.
 Then we have 
 \[
n\leq \lim_{r\to\infty}\ n_o'(r)
\leq 
n\frac{2+ k}{2}.
\]
In order to see the upper estimate on the limit of $n'$, we estimate
\begin{multline*}
\lim_{r\to\infty}\ n'_o(r)=
 \lim_{r\to\infty}\ n\cdot 1\vee \ln\left[(A(r)^2r^2(1\vee \|\deg\|_{B_o(A(r)r)})\right]^{{\iota(r)}+\vartheta(r)}
\\
 \leq n\cdot 1\vee (k+2)\lim_{r\to\infty}\left[\ln(A(r))^{\iota(r)+\vartheta(r)}+\ln r^{{\iota(r)}}+\ln r^{\vartheta(r)}\right].
\end{multline*}
In order to estimate the first limit, note that $$ \|\deg\|_{B_o(r)}^{\vartheta(r)}\leq C^{\vartheta(r)} r^{k\vartheta(r)}=C^{\vartheta(r)} \exp\left(\frac{k}{(\ln r)^{\frac{1}{n+2}}}\right)\longrightarrow 1,\qquad r\to \infty,$$
such that, since $\iota(r),\theta(r)\to 0$ as $r\to\infty$, 
\[
\ln(A(r))^{\iota(r)+\vartheta(r)}\leq (\iota(r)+\vartheta(r))(2^{19n}\euler C\cdot 1\vee \|\deg\|_{B_o(r)}^{\vartheta(r)})\longrightarrow 0,\qquad r\to\infty.
\]
For the second limit we use $\iota(r)=\frac{1}{2\ln \left(\frac{r}{(\ln r)^{n+3} }\right)}=\frac12\frac{1}{\ln r-(n+3)\ln{\ln r }}$ to obtain
 \[
\ln r^{\iota(r)}
=\frac12\cdot \frac{1}{1-(n+3)\frac{\ln\ln r}{\ln r}}
\longrightarrow\frac12, \qquad r\to \infty.
 \]
 The third limit satisfies
\[
\ln r^{\vartheta(r)}
= C_{n,S}(\ln r)^{1-\frac{n+3}{n+2}}\longrightarrow 0,\qquad r\to \infty.
\]
We arrive at
\[
\lim_{r\to\infty}\ n_o'(r)
\leq 
n\cdot 1\vee(k+2)\lim_{r\to\infty}\left[\ln r^{{\iota(r)}}+\ln r^{\vartheta(r)}\right]=n\cdot 1\vee \frac{k+2}{2}=n\frac{2+k}{2}.
\]
In conclusion, if $k>0$, the dimension $n'$ may be affected by the behavior of $\|\deg\|_{B_o(r)}$. 
If $\|\deg\|_{B_o(r)}$ grows exponentially, the dimension $n'$ might become unbounded.
 \end{remark}

Finally, we present the case of general measure and intrinsic path metric with jump size $S>0$. 
\\
The existence of a variable dimension in the above result for the counting measure suggests to include this variability in both implications of the characterization. Therefore, we need to modify the error functions introduced above for the counting measure to include variable dimensions.
For $R\geq 288S$, $r> 0$, $x,y\in X$, and a dimension function $n\colon X\times [R,\infty)\to (0,\infty)$, let 
\[
\nu_{x,n(R)}=[1\vee\Deg_x]^{\frac{n_x(R)}{2}},\qquad 
\Phi_x^{n(R)}(r)=\begin{cases}
\nu_{x,n(R)}^{\theta(n_x(R), R_1)}R_1^{n_x(R_1)}\nu_{x,n(R_1)} &\colon r<R_1,\\[1.5ex]
\nu_{x,n(R)}^{\theta(n_x(R),r)} &\colon \text{else},
\end{cases}
\]
and for $\tau> 0$
\[
\Psi_{xy}(\sqrt \tau)^2=\Phi_x^{n(\sqrt\tau)}(\sqrt{\tau_\rho})\ \Phi_y^{n(\sqrt\tau)}(\sqrt{\tau_\rho}).
\]
For $R'=\euler^{1\vee 8R}$ and $C\colon[R',\infty)\to[1,\infty)$, let a second dimension function $n'\colon X\times [R',\infty)\to (0,\infty)$ be given by
\[
n_x'(r)=
\|n\|_{Q_x(R,r)}\cdot 1\vee \ln \left[A(r)^2r^2(1\vee \|\Deg\|_{B_x(A(r)r)})\right]^{\iota(\|n\|_{Q_x(R,A(r)r)},r)+\vartheta(\|n\|_{Q_x(R,A(r)r)},r)},
\]
where 
\[A(r)=\exp\left(2^{19\|n\|_{Q_x(R,A(r)r)}}\euler \|C\|_{Q_x(R,A(r)r)}\cdot 1\vee\|\deg\|_{B_o(r)}^{\vartheta(\|n\|_{Q_x(R,A(r)r)},r)}\right),\]
and where we denote space-time cylinders by $$Q_x(s,t):=B_x(t)\times [s,t],\qquad s,t\in\RR,\ x\in X.$$
Analogously to the case of the counting measure, denote
\[ a'(r)=\left(2^{36}\|C\|_{Q_x(R,A(r)r)}^{26/\|n\|_{Q_x(R,A(r)r)},r)}(\ln A(r))^{2/\|n\|_{Q_x(R,A(r)r)}}{A(r)^2}\right)^{-1}.
\]
With these notions at hand, the main result for general measures is the following.
\begin{theorem}[general measure]\label{thm:generalmainglobal}There exists $R_0=R_0(S)>0$ such that for all $R\geq R_0$
we have the following.
\begin{enumerate}[(i)]
\item If there are real functions $a,n>0$ such that $FK(R,a,n)$ holds in $X$, then there exists a real function $C=C_{a,n,S}\geq1$ such that $G(R,C\Psi,n)$, $L(R,C\nu,n) $, and $V(C\Phi,n)$ hold in $X$.
\item 
Conversely, if there are real functions $C,n>0$ such that $G(R,C\Psi,n)$, $L(R,C\nu,n) $, and $V(C\Phi,n)$ hold in $X$, then we have $FK(8R',a',n')$ in $X$. 
\end{enumerate}
In particular, $R_0(S)\to0$ and $C_{a,n,S}\to C_{a,n,0}<\infty$ for $S\to0$.
\end{theorem}

\subsection{Strategy of the proofs}
Our results follow from variants of the localized versions of the implications of Theorem~\ref{thm:generalmainglobal}, which contain even more precise statements and allow for a larger range of parameters. The proofs of these theorems are purely analytic and greatly expand on the general strategy in \cite{Grigoryan-94} using a variety of ideas. 
\\

We start with 
the implication (FK) $\Rightarrow$ (L) $\&$ (V) $\&$ (G).

Lemma~\ref{lem:localreg} contains the simple implication (FK) $\Rightarrow$ (L). Local regularity follows from the Faber-Krahn inequality by plugging in a characteristic function of a vertex.

Theorem~\ref{thm:asympdoubling} shows (FK) $\Rightarrow$ (V). By plugging in suitable test functions we can use the classical iteration argument found in, e.g., \cite{SaloffC-01}, cf.~\cite{KellerRose-24}. Test functions with a suitable control on their gradients are provided by the existence of the intrinsic metric. The iteration argument works as long as the supports of the gradients of the test functions shrink in each iteration step. It stops due to  the discrete structure of the space, resulting in an error incorporated in the doubling constant. By controlling this error via the local regularity property (L) it only depends on the vertex degree of the center of the ball. On large scales, this error converges to one. Our proofs of Lemma~\ref{lem:localreg} and Theorem~\ref{thm:asympdoubling} can be found in Section~\ref{section:regularitydoubling}.

Theorem~\ref{thm:gaussianboundclean} implements the implication (FK) $\Rightarrow$ (G). Its proof has three main ingedredients: an $\ell^2$-mean value inequality, a new integrated maximum principle, and their combination via a new variant of Davies' method.
\\
An $\ell^2$-mean value inequality for positive parts of subsolutions of the heat equation is derived in Section~\ref{section:tas}. The proof splits into the following subsections according to the steps of the proof.
\\
Subsection~\ref{subsection:maximalinequalities} proves maximal inequalities for positive parts of subsolutions of the heat equation. This done by combining \cite[Lemma~2.2]{KellerRose-22a}, which uses the intrinsic metric property, and an adapted simplified version of \cite[Lemma~4.2]{GrigoryanHH}.
\\
Subsection~\ref{subsection:deGiorgielementary} is dedicated to the derivation of  elementary de Giorgi-Grigor'yan iteration steps based on an elaboration on \cite{GrigoryanHH}. Lemma~\ref{lem:elementarystep} shows that the combination of a good cut-off function provided by the intrinsic metric, the Faber-Krahn inequality, and the above maximal inequality suffices to obtain the iteration step for large enough balls. Additionally, the maximal inequality also leads to a $\ell^2$-mean value inequality which always holds on graphs even without (FK), see Lemma~\ref{lem:elementarypoint}.
\\
Subsection~\ref{lem:elementarypoint} provides the mean value inequality by Grigor'yan's version of de Giorgi iteration by using Lemma~\ref{lem:elementarystep} and Theorem~\ref{thm:asympdoubling} assuming Faber-Krahn inequalities. Due to the discrete structure and the lack of cut-off functions with controlled gradient for small radii, the iteration procedure naturally stops after a certain amount of steps. 
We derive an estimate for powers of norms of positive parts of subsolutions in a ball in terms of the norm in an enlarged ball, where the power depends on the number of iteration steps. In order to obtain the desired mean value inequality, we must use the estimate from Lemma~\ref{lem:elementarypoint} which is always satisfied on graphs with intrinsic metric. This  results in an error for the constant in the mean value inequality depending on the inverse of the measure of the considered vertex and the constant in the Faber-Krahn inequality. In particular, the error becomes small if the radius is large.

Section~\ref{section:intmaxprinciple} provides a new integrated maximum principle for graphs, Theorem~\ref{thm:intmaxprinciple}. If Lipschitz functions satisfy a new discretized version of the Hamilton-Jacobi inequality suitable for graphs, we can perform the calculations in the Davies-Gaffney-Grigor'yan lemma, cf.~\cite{Grigoryan-09,BauerHuaYau-17}. Incorporating properties of the intrinsic metric viz ideas from \cite{BauerHuaYau-17} eventually allows to derive the sharp Gaussian for the heat kernel.

Section~\ref{section:weightedl2norm}
deals with graph versions of weighted $\ell^2$-norms of the heat kernel in terms of specific solutions of the discretized Hamilton-Jacobi equation provided in Section~\ref{section:intmaxprinciple}. We bound these norms by combining our new integral maximum principle for these solutions with the mean-value inequalities derived from the Faber-Krahn inequality. In order to obtain a reasonable bound for all times, we again have to modify the argument by incorporating a discrete version of the mean value inequality which is satisfied on any graph.

Section~\ref{section:FKhkb} provides localized and optimal Gaussian upper heat kernel bounds assuming relative Faber-Krahn inequalities in large balls, cf.~Theorem~\ref{thm:gaussianboundclean}. Especially, the bottom of the spectrum of the Laplacian is involved in the upper bound. The proof is based on Theorem~\ref{thm:split}, a new combination of Davies' method \cite{Davies-87} and the upper bound of the heat kernel in terms of weighted $\ell^2$-norms.
In fact, choose initial solutions of the Hamilton-Jacobi equation in Section~\ref{section:weightedl2norm} and bound the heat kernel by the associated weighted $\ell^2$-norms. Optimization over the weight provides the optimal Gaussian. Finally, we apply the results from Section~\ref{section:weightedl2norm} to conclude and use the local regularity property (L) to obtain an error depending only on the vertex degree.
\\

The implication (L) $\&$ (V) $\&$ (G) $\Rightarrow$ (FK) is the content of  Theorem~\ref{thm:FKlocalregclean}. Its proof is the content of Section~\ref{section:gausstoFK} and follows the steps below.

We start with the derivation of the reverse volume doubling property on large balls from (V), Lemma~\ref{lem:reversedoubling} by a well-known chaining argument. Note that this is the only part of the proof which requires the intrinsic metric to be a path metric.

Next, we provide a  variant of a general lower bound for the spectral bottom of the Dirichlet Laplacian of bounded subsets following (V) and (G), Lemma~\ref{lem:FKapriori}. Due to the restricted range of radii and times where (V) and (G) hold, this lower bound depends on an unpleasant error in terms of reciprocals of measures. This unpleasant error is still good enough in very bounded situations such as the normalizing measure, cf.~Theorem~\ref{thm:normalizedclassicalFK}. However, allowing for more unboundedness we need a new idea which is the choice of a variable dimension and the incorporation of the local regularity property (L). This choice mitigates the possibly unbounded error terms for large radii, cf.~Theorem~\ref{thm:fixeddimension} which still includes a free parameter which is chosen later to prove Theorems~ \ref{thm:generalmainglobal} and in turn Theorem \ref{thm:countingmainglobal}. 
\\

Finally, we give the proofs of the main theorems. Theorem~\ref{thm:normalizedmain} is the global version of Corollaries~\ref{cor:asympdoublingnormalized},~\ref{cor:gaussianboundcleannormalized}, and Theorem~\ref{thm:normalizedclassicalFK}. Theorem~\ref{thm:generalmainglobal} is the global version of Theorems~\ref{thm:asympdoubling},~\ref{thm:gaussianboundclean}, and~\ref{thm:FKlocalregclean}. Theorem~\ref{thm:countingmainglobal} is the specialization of Theorem~\ref{thm:generalmainglobal} to the counting measure.

\section{Faber-Krahn, local regularity and doubling}\label{section:regularitydoubling}
In this section, we show that the Faber-Krahn inequality (FK) on large balls yields the local regularity property (L) and the volume doubling property (V) on large balls. The error incorporating the local geometry on graphs will be given in terms of the function
\[
\nu_o(n_o(R))=
[1\vee \Deg_o]^{\frac{n_o(R)}{2}}.
\]

First, we derive property (L) from (FK). As usual, in what follows we use the convention
\[
\sum_Um f:=\sum_{x\in U}m(x)f(x),\quad f\in\cC(U),\ U\subset X.
\]
\begin{lemma}[Local regularity]\label{lem:localreg} Let $o\in X$ and constants $a,n>0$, $R\geq S$ such that $FK(R,R,a,n)$ holds in $o$. 
Then there is a constant $C=C_{a,n}>0$ such that $L(R,R,C\nu,n)$ holds in $o$.
More generally, if there are $R_2\geq R_1\geq S$ and functions $a,n\colon [R_1,R_2]\to (0,\infty)$ such that $FK(R_1,R_2,a,n)$ holds in $o$, 
then there is a function $C=C_{a,n}>0$ such that $L(R_1,R_2,C\nu, n)$ holds in $o$. 
\end{lemma}
\begin{proof}
Since $R\geq S$ we can apply the Faber-Krahn inequality to $U=\{o\}$ and obtain
\[
\frac{a}{R^2}\left(\frac{m(B_o(R))}{m(o)}\right)^{\frac{2}{n}}
\leq \frac{\sum_{B_o(R)} m\vert \nabla \Eins_{\{o\}}\vert^2}{\sum_{B_o(R)} m\Eins_{\{o\}}^2}\leq 4\Deg_o,
\]
what implies the first claim with $C=\left({4}/{a}\right)^{\frac{n}{2}}$. The general statement is an immediate consequence.
\end{proof}

In the special case of the normalizing measure, Lemma~\ref{lem:localreg} reduces to the following.

\begin{corollary}[Local regularity normalizing measure]
\label{cor:localregnormalized}Let $m=\deg$, $o\in X$ and constants $a,n>0$, 
$R_2\geq R_1\geq S$, such that $FK(R_1,R_2,a,n)$ holds in $o$.
Then there is a constant $C=C_{a,n}>0$ such that $L(R_1,R_2,C, n)$ holds in $o$. 
\end{corollary}
\begin{proof}
This is an immediate consequence of Lemma~\ref{lem:localreg} and $\Deg_o=1$.
\end{proof}

Next, we turn to the derivation of the volume doubling property (V) from (FK).
A key ingredient in everything what follows is the existence of cut-off functions with respect to intrinsic metrics. For such a metric $ \rho $, $ A\subset X $, and $ R\ge0 $ we let  $ B_A(R)=\{x\in X\colon \inf_{a\in A}\rho(x,a)\leq R \} $.
\begin{proposition}[{\cite[Proposition~11.29]{KellerLW-21}}]\label{prop:cutoff} Assume $A\subset X$, $R\geq 0$, and set
\[
\phi_{A,R}=\left(1-\frac{\rho(\cdot,A)}{R}\right)_+.
\]
Then, we have $1_{A}\leq \phi_{A,R} \leq 1_{B_{A}(R)}$ and 
\[
\Vert \vert \nabla\phi_{A,R}\vert\Vert_\infty^2\leq \frac1{R^2}.
\]
\end{proposition} 
In order to derive (V) in our general context, let $0\leq R_1\leq R_2$ and 
for $n\colon [R_1,R_2]\to (0,\infty)$ and $o\in X$, recall for $0<r\leq R\leq R_2$
\[
\Phi_o^{n(R)}(r)=
\begin{cases}
\nu_{o,n(R)}^{\theta(n(R), R_1)}R_1^{n(R_1)}\nu_{o,n(R_1)} &\colon r<R_1,\\[1.5ex]
\nu_{o,n(R)}^{\theta(n(R),r)} &\colon \text{else},
\end{cases}
\qquad
 \theta(n,r)=\frac{n+2}{n}\cdot 1\wedge \left(\frac{288S}{r}\right)^{\frac{1}{n+2}}.
\]

\begin{theorem}\label{thm:asympdoubling}
Assume that there are $o\in X$ and constants $n>0$, $a\in(0,1]$ and $R\geq 2S$ such that $FK(R,R,a,n)$ holds in $o$.
Then there is a constant $C=C_{a,n,S}>0$ such that for all $r\in[2S,R]$, we have 
\begin{align*}
\frac{m(B_o(R))}{m(B_o(r))}\leq C\Phi_o^n(r) \left(\frac{R}{r}\right)^{n}.
\end{align*}
More generally, if $R_2\geq R_1\geq 2S$, $n\colon [R_1,R_2]\to (0,\infty)$, and $a\colon [R_1,R_2]\to (0,1]$ such that $FK(R_1,R_2,a,n)$ holds in $o$, then there is a function $C=C_{a,n,S}>0$ such that  $V(R_1,R_2,C\Phi,n) $ holds in $o$.
\end{theorem}
\begin{proof}Let $r\leq R$ and abbreviate $B(r):=B_o(r)$, $m(r):=m(B_o(r))$. 
Choose
\[
\phi=\phi_0=(r-S-\rho(o,\cdot))_+.
\]
Proposition~\ref{prop:cutoff} yields $\supp \phi_0\subset B(r-S)
\subset B(r)$, $\vert\nabla\phi_0\vert\leq 1$, and $\supp \vert\nabla\phi_0\vert\subset B(r)$.
Since $2S\leq r$ we have $(r-S)/2\geq (r-r/2)/2=r/4 $. Thus, we get
\[
\sum_{B(r)}m\phi_0^2\geq \left(\frac{r-S}{2}\right)^2m((r-S)/2)\geq\left(\frac{r}{4}\right)^2m(r/4). 
\]
By the definition of the first Dirichlet eigenvalue we obtain
\[
\lambda(B(r))=\inf_{\supp \phi\in B(r)}\frac12\frac{\sum_{B(r)} m\vert \nabla \phi\vert^2}{\sum_{B(r)} m\phi^2}\leq \frac{\sum_{B(r)} m\vert \nabla \phi_0\vert^2}{\sum_{B(r)} m\phi_0^2}
\leq \left(\frac{4}{r}\right)^2\frac{m(r)}{m(r/4)}.
\]
The Faber-Krahn inequality $FK(R,R,a,n)$ applied to $U=B(r)$ yields 
\[
\frac{a}{R^2}\left(\frac{m(R)}{m(r)}\right)^{\frac{2}{n}}
\leq \lambda (B(r))\leq \left(\frac{4}{r}\right)^2\frac{m(r)}{m(r/4)}=\frac{2^4}{r^2}\frac{m(r)}{m(r/4)},
\qquad\text{i.e.,}\qquad
\left(\frac{a}{2^4}\right)^q\left(\frac{r}{R}\right)^{2q}
m(R)^{{\frac{2}{n}} q} 
m(r/4)^{q}
\leq 
m(r).
\]
where we set $q:=1/(1+{\frac{2}{n}})$.
If we set $k_0=\lfloor \log_2\sqrt{r/S}\rfloor$, we have 
\[
\frac{r}{4^{k_0}}\geq \frac{r}{4^{ \log_2\sqrt{r/S}}}=S.
\]
Hence, iterating the above inequality for $k\in\{0,\ldots, k_0\}$ yields
\[
\left(\frac{a}{2^4}\right)^{\sum_{i=1}^{k_0}q^i}
\left(\frac{r}{R}\right)^{2{\sum_{i=1}^{k_0}q^i}}
m(R)^{{\frac{2}{n}}\sum_{i=1}^{k_0}q^i}
m(r/4^{k_0})^{q^{k_0}}
\leq m(r).
\]
We estimate the factors on the left-hand side and start with the evaluation of the sum appearing in the exponents. Since $q/(1-q)=n/2$, we have
\[
\sum_{i=1}^{k_0}q^i
=
\sum_{i=0}^{k_0}q^i -1
=
\frac{q(1-q^{k_0})}{1-q}
=
\frac{n}2(1-q^{k_0}).
\]
Since  $(1-q^{k_0})\geq 0$ and $a\leq 1$, this yields for the first two factors 
\[
\frac{a^{\frac{n}{2}}}{2^{2n}}
\leq \left(\frac{a}{2^4}\right)^{\frac{n}{2}(1-q^{k_0})}=\left(\frac{a}{2^4}\right)^{\sum_{i=1}^{k_0}q^i}\qquad\text{and}\qquad
\left(\frac{r}{R}\right)^{n}
\left(\frac{R}{r}\right)^{nq^{k_0}}
=\left(\frac{r}{R}\right)^{n\left(1-q^{k_0}\right)}
=\left(\frac{r}{R}\right)^{2{\sum_{i=1}^{k_0}q^i}}.
\]
Since $r/4^{k_0}\geq S$ we also have $m(r/4^{k_0})\geq m(o)$. Applying these relations, we obtain
\[
\frac{a^{\frac{n}{2}}}{2^{2n}}
\left(\frac{r}{R}\right)^{n}\left(\frac{R}{r}\right)^{nq^{k_0}}
m(R)^{1-q^{k_0}}
m(o)^{q^{k_0}}
\leq m(r),
\qquad\text{
i.e.,}
\qquad
\frac{m(R)}{m(r)}
\leq 
\frac{2^{2n}}{a^{\frac{n}{2}}}
\left(\frac{R}{r}\right)^{n}
\left(
r^n\frac{m(R)}{m(o)R^n}
\right)^{q^{k_0}}.
\]
By Lemma~\ref{lem:localreg}, $FK(R,R,a,n)$ in $o$ implies $L(R,R,C\nu,n)$ in $o$ with $C=\left(\frac{4}{a}\right)^{\frac{n}{2}}$, i.e.,
\[
\frac{m(R)}{m(o)}
\leq \left(\frac{4}{a}\right)^{\frac{n}{2}}\nu_o(R)\ R^n.
\]
Plugging this inequality into the above estimate on $m(R)/m(r)$,  we arrive at
\[
\frac{m(R)}{m(r)}
\leq 
\frac{2^{2n}}{a^{\frac{n}{2}}}
\left(\frac{R}{r}\right)^{n}
\left(
r^n\left(\frac{4}{a}\right)^{\frac{n}{2}}\nu_o(R)
\right)^{q^{k_0}}
=\frac{2^{2n+nq^{k_0}}}{a^{\frac{n}{2}(1+q^{k_0})}}
r^{nq^{k_0}}\nu_o(R)^{q^{k_0}}
\left(\frac{R}{r}\right)^{n}
.
\]
Using $k_0=\lfloor \log_2\sqrt{r/S}\rfloor\geq \log_2\sqrt{r/S}-1$ and $q=n/(n+2)\leq 1$, we have 
$$q^{k_0}\leq q^{\log_2\sqrt{r/S}-1}
=\frac{1}{q}\left(\frac{S}{r}\right)^{\log_2\sqrt{\frac1q}}=\frac{n+2}{n}\left(\frac{S}{r}\right)^{\log_2\sqrt{\frac{n+2}n}}.
$$
The mean value theorem and $\ln2\le1$ yield
\[
\log_2\sqrt{(n+2)/n}=\frac{1}{2\ln2}(\ln(n+2)-\ln n)\geq
\inf_{x\in[n,n+2]}\frac{1}{x}=\frac{1}{n+2},
\]
such that, since $r\geq 2S$ implies $S/r\leq 1$, we get
\[
q^{k_0}\leq \frac{n+2}{n}\left(\frac{S}{r}\right)^{\frac{1}{n+2}}\leq \frac{n+2}{n}\cdot 1\wedge \left(\frac{288S}{r}\right)^{\frac{1}{n+2}}=\theta(n,r).
\]
Setting $f(x)=x^{\left(1/x\right)^{\log_2\sqrt{\frac1q}}} $, we obtain, using $n/q=n(n+2)/n=n+2$,
\[
r^{nq^{k_0}}\leq1\vee  r^{\frac{n}{q}\left(\frac{S}{r}\right)^{\log_2\sqrt{\frac1q}}}\leq 1\vee S^{(n+2)\left(\frac{S}{r}\right)^{\log_2\sqrt{\frac1q}}} f(r/S)^{(n+2)}.
\]
Since $r\geq S$, we have $S^{(n+2)\left(\frac{S}{r}\right)^{\log_2\sqrt{\frac1q}}}\leq (1\vee S)^{n+2}$.
The function $f$ attains its maximum at the value $x_0=\exp(1/\log_2\sqrt{1/q})$ with maximum $f(x_0)=\exp(1/(\euler\log_2\sqrt{1/q}))=2^{2/(\euler\ln(1/q))}\leq 2^{1/\ln(1/q)}\leq 2^{n+2}$, where we used the mean value theorem in the last step.
Therefore, we have 
\[
r^{q^{k_0}}\leq (1\vee S)^{n+2}2^{(n+2)^2}=(1\vee S)^{n+2}2^{n^2+4n+4}.\]
Plugging this estimate into the above bound on $m(R)/m(r)$, using $a,q\leq 1$, and recalling the definition $m(r)=m(B(r))$ yields the result with $C_{a,n,S}=(1\vee S)^{n+2}\left(\frac{2}{a}\right)^{n^2+7n+4}$. The general statement follows right away with $C_{a,n,S}=(1\vee S)^{n(R)+2}\left(\frac{2}{a}\right)^{n(R)^2+7n(R)+4}$.
\end{proof}

Gluing together properties (L) and (V) for a restricted amount of radii leads to property (V) on all small enough scales.
\begin{lemma}\label{lem:fulldoubling}Let $o\in X$, $R_2\geq R_1\geq 2S$, $n\colon [R_1,R_2]\to (0,\infty)$, and functions $C_1,C_2\colon [R_1,R_2]\to(0,\infty)$ such that $V(R_1,R_2,C_1\Phi,n) $ and $L(R_1,R_2,C_2\nu,n)$ hold in $o$. Then there exists a function $C'=C'_{C_1,C_2}>0$ such that $V(R_2,C'\Phi,n)$ holds in $o$.
\end{lemma}
\begin{proof}We must show that volume doubling holds for radii smaller than $R_1$ and two radii sandwiching $R_1$.
We infer from $L(R_1,R_2,C_1\nu,n)$ that for all $r\in[R_1,R_2]$ we have
$
{m(B_o(r))}/{m(o)}\leq C_{1}(r)\nu_o(n(r))r^{n(r)}.
$
Therefore, for all $0< r_1\leq r_2\leq R_1$, we have
\[
\frac{m(B_o(r_2))}{m(B_o(r_1))}\leq \frac{m(B_o(R_1))}{m(o)}\leq C_{1}(R_1)\nu_o(n(R_1))R_1^{n(R_1)}.
\]
Further, if $0<r_1\leq R_1\leq r_2\leq R_2$, we use $V(R_1,R_2,C_2\Phi,n))$ in $o$ and the above estimate to obtain
\begin{multline*}
\frac{m(B_o(r_2))}{m(B_o(r_1))}
\leq \frac{m(B_o(r_2))}{m(o)}
\leq C_{2}(r_2)\Phi^{n(r_2)}(R_1)\left(\frac{r_2}{R_1}\right)^{n(r_2)}\frac{m(B_o(R_1))}{m(o)}
\\
\leq 
C_{2}(r_2)C_{1}(R_1)\nu_o(n(R_1))R_1^{n(R_1)}
\Phi^{n(r_2)}(R_1)\left(\frac{r_2}{r_1}\right)^{n(r_2)},
\end{multline*}
and this yields the claim.
\end{proof}
\begin{theorem}\label{thm:fulldoubling}
Let $o\in X$, $R_2\geq R_1\geq 2S$, $n\colon [R_1,R_2]\to (0,\infty)$, and $a\colon [R_1,R_2]\to (0,1]$ such that $FK(R_1,R_2,a,n)$ holds in $o$. Then there exists a function $C=C_{a,n,S,R_1}>0$ such that  $V(R_2,C\Phi,n) $ holds in $o$.
\end{theorem}
\begin{proof}By Theorem~\ref{thm:asympdoubling}, we infer the existence of a function $C_{a,n,S}>0$ such that  $V(R_1,R_2,C_{a,n,S}\Phi,n) $ holds in $o$. Hence, Lemma~\ref{lem:localreg} yields the existence of a function $C_{a,n}>0$ such that $L(R_1,R_2,C_{a,n}\nu,n)$ holds in $o$. Hence, the statement follows from Lemma~\ref{lem:fulldoubling}.
\end{proof}
In the special case of the normalizing measure, Theorem~\ref{thm:fulldoubling} reduces to the following.

\begin{corollary}[Volume doubling normalizing measure]\label{cor:asympdoublingnormalized}
Let $m=\deg$, and assume that there are $o\in X$ and constants $n>0$, $a\in(0,1]$, and $R_2\geq R_1\geq 2$ such that $FK(R_1,R_2,a,n)$ holds in $o$.
Then there is a constant $C=C_{a,n,R_1}>0$ such that  $V(R_2,C,n)$ holds in $o$.
\end{corollary}
\begin{proof}Since $m=\deg$, we have $\nu=1$ and hence $\Phi\leq 1\vee R_1^n$. Since $\rho$ is the combinatorial distance with $S=1$, and $n$ is a constant, Theorem~\ref{thm:asympdoubling} yields the claim with $C_{a,n,R_1}=\left(\frac{2}{a}\right)^{n^2+7n+4}(1\vee R_1^n)$.
\end{proof}

\section{Mean value inequality}\label{section:tas}

This section is devoted to the derivation of the $\ell^2$-mean value inequality for solutions of the heat equation. After obtaining maximal inequalities for regularized versions of positive parts of subsolutions, we run the deGiorgi-Grigor'yan iteration. This iteration procedure naturally stops after an amount of steps given in terms of the radius of the ball under consideration. In order to obtain the mean value inequality, we incorporate a new discrete version of the iteration step which always holds on a graph with intrinsic metric.

Let $I\subset\RR$ and $B\subset X$. The function $u\colon I\times B\to\RR$ is called solution (resp.~subsolution) on $I\times B$ if  for any $x\in B$ the map $t\mapsto u_t(x):=u(t,x)$ is continuously differentiable in the interior of $I$ such that the differential has a continuous extension to the closure of $I$ 
and satisfies
\[
\left(\frac{d}{dt}+\Delta \right)u =0\quad\text{on}\ I\times B\quad \mbox{(resp.~$ \leq$)}.
\]

\subsection{Maximal inequalities for positive parts of  subsolutions}\label{subsection:maximalinequalities}

Here, we derive maximal inequalities for positive parts of subsolutions of the heat equation which are a necessary tool to carry out the de Giorgi iteration step. We will prove such inequalities for regularized subsolutions by transferring  the arguments in \cite{GrigoryanHH} to the graph setting and combining them with a maximal inequality for subsolutions obtained in \cite[Lemma~2.3]{KellerRose-22a}.
\begin{lemma}\label{lem:ffunctions}
Let $T_1\leq T_2$, $B\subset X$ finite, $u\geq 0$ be a non-negative subsolution on $[T_1,T_2]\times B$, and 
 $F\in\cC^1(\RR)$ with
 \[
 F\restriction_{(-\infty,0]}=0,\quad F'\geq 0, \quad\text{and}\quad
 \sup_\RR F'<\infty.
 \]
Then for any $\epsilon\geq 0$ the function
\[
v:=F\circ (u-\epsilon)_+
=F\circ(u-\epsilon)
\]
is a subsolution on $[T_1,T_2]\times B$.
\end{lemma}

\begin{proof}
Define 
\[
F_\epsilon\colon \RR\longrightarrow\RR, \quad x\longmapsto F(x-\epsilon)
\]
and observe that $F_\epsilon$ is differentiable and vanishes at zero. Since we have $v=F_\epsilon\circ u$, the chain rule yields
\[
\frac{d}{dt}v(x)=F_\epsilon'(u(x))\cdot \frac{d}{dt}u_t(x).
\]
Since $F'\geq 0$ we have $F_\epsilon'\geq 0$ such that $F_\epsilon$ is convex. Hence, we have 
\begin{multline*}
\Delta v(x)
=\Delta (F_\epsilon\circ u)(x)
=
\frac{1}{m(x)}\sum_{y\in X
}b(x,y)(F_\epsilon(u(x))-F_\epsilon(u(y)))
\\
\leq 
\frac{1}{m(x)}\sum_{y\in X
}b(x,y)F_\epsilon'(u(x))(u(x)-u(y))
=F_\epsilon'(u(x))\Delta u(x).
\end{multline*}
Since $F_\epsilon'\geq 0$, the subsolution property of $u$ yields
\[
\frac{d}{dt}v(x) +\Delta v(x)
\leq F_\epsilon'(u(x))\left(\frac{d}{dt} u(x)+\Delta u(x)
\right)\leq 0,
\]
i.e., $v$ is a subsolution.
\end{proof}

The combinatorial interior of a set $ B\subset X $ is given by
\begin{align*}
	B^{\circ}=\{x\in B\colon b(x,y)=0 \mbox{ for all }y\in X\setminus B \}.
\end{align*}
We have the following maximal inequality for all subsolutions of the heat equation.
\begin{proposition}[{\cite[Lemma~2.3]{KellerRose-22a}}]\label{prop:KRa}
Let $T_1\leq T_2$, $B\subset X$ finite, functions $\phi\colon X\to[0,\infty)$, $ 0\leq \phi\leq \mathbf{1}_{B^{\circ}}  $,
 $\chi\colon [T_1,T_2]\to \RR$ piecewise differentiable, 
\[\eta\colon [T_1,T_2]\times B\to\RR,\qquad (t,x)\mapsto\eta_t(x):=\chi(t)\phi(x).
\] 
Then we have for all subsolutions $v$ and $t\in[T_1,T_2]$
\[
\frac{d}{dt}\|\eta_tv_t\|^{2}+\frac{1}{2}\left\|\eta_t\vert\nabla v_t\vert\right\|^2
\leq \Vert\Eins_Bv_t\Vert^2
\left(2\eta_t\frac{d}{d t}\eta_t+166\|\vert\nabla \eta_t\vert\|_{X}^{2}\right).
\]
\end{proposition}

The combination of Lemma~\ref{lem:ffunctions} and Proposition~\ref{prop:KRa} leads to the following maximal inequality for regularized positive parts of subsolutions.
\begin{lemma}\label{lemma:maxandintgeneral}Let $T_1\leq T_2$, $B\subset X$ finite, $\phi\colon X\to[0,\infty)$, $ 0\leq \phi\leq \mathbf{1}_{B^{\circ}}  $,
 $\chi\colon [T_1,T_2]\to \RR$ piecewise differentiable, 
\[\eta\colon [T_1,T_2]\times B\to\RR,\qquad (t,x)\mapsto\eta_t(x):=\chi(t)\phi(x).
\] 
Then we have for $v$, $\epsilon$, and $F$ as in Lemma~\ref{lem:ffunctions} and $t\in[T_1,T_2]$
\[
\frac{d}{dt}\|\eta_tv_t\|^{2}+\frac{1}{4}\Vert\vert\nabla (\eta_t v_t)\vert\Vert^2
\leq \sup_{\RR}F'\cdot \Vert\Eins_B(u_t-\epsilon)_+\Vert^2
\left(2\eta_t\frac{d}{d t}\eta_t+167\|\vert\nabla \eta_t\vert\|_{X}^{2}\right).
\]
\end{lemma}

\begin{proof}
We bound the left-hand side in the statement in terms of the left-hand side of 
Proposition~\ref{prop:KRa}.
To ease notation, we let $\av_{xy}(\phi):=\tfrac12(\phi(x)+\phi(y))$ and recall $\nabla_{xy}\phi=\phi(x)-\phi(y)$ for $x,y\in X$ and $\phi\in\cC(X)$. Then we have  $\nabla_{xy}(\phi v)=\av_{xy}(\phi)\nabla_{xy}v+\av_{xy}(v)\nabla_{xy}\phi$ and hence, using $2ab\leq a^2+b^2$,
\[
(\nabla_{xy}(\phi v))^2=(\av_{xy}(\phi)\nabla_{xy}v+\av_{xy}(v)\nabla_{xy}\phi)^2
\leq 2\av_{xy}(\phi)^2(\nabla_{xy}v)^2 +2 \av_{xy}(v)^2(\nabla_{xy}\phi)^2.
\]
Therefore,
\[
\Vert \vert \nabla (\phi v)\vert\Vert^2
=\sum_{x,y\in X}b(x,y)(\nabla_{xy}(\phi v))^2
\\\leq 2\sum_{x,y\in X}b(x,y)\av_{xy}(\phi)^2(\nabla_{xy}v)^2 +2 \sum_{x,y\in X}b(x,y)\av_{xy}(v)^2(\nabla_{xy}\phi)^2.
\]
To bound the first summand, recall  $\av_{xy}(\phi)=\frac{1}{2}(\phi(x)+\phi(y))$ and use $2ab\leq a^2+b^2$ to obtain
\[
2\sum_{x,y\in X}b(x,y)\av_{xy}(\phi)^2(\nabla_{xy}v)^2
\leq 
2 \Vert \phi\vert \nabla v\vert \Vert_2^2.
\]
Now we estimate the second summand. The definition of $\av_{xy}$ and $\supp \nabla \phi\subset B\times B$ imply
\[
2 \sum_{x,y\in X}b(x,y)\av_{xy}(v)^2(\nabla_{xy}\phi)^2
\leq 2\Vert \vert \nabla \phi\vert\Vert_\infty^2 \sum_Bm\ v^2=2\Vert \vert \nabla \phi\vert\Vert_\infty^2 \Vert\Eins_Bv\Vert^2.
\]
Hence, recalling $\eta_t=\chi(t)\phi$, 
\[
\Vert \vert \nabla (\eta_t v)\vert\Vert^2
=\chi(t)^2\Vert \vert \nabla (\phi v)\vert\Vert^2
\leq 2 \chi(t)^2\Vert \phi\vert \nabla v\vert \Vert^2+2\chi(t)^2\Vert \vert \nabla \phi\vert\Vert_X^2 \Vert\Eins_Bv\Vert^2
=2 \Vert \eta_t\vert \nabla v\vert \Vert^2+2\Vert \vert \nabla \eta_t\vert\Vert_X^2 \Vert\Eins_Bv\Vert^2.
\]
Using Proposition~\ref{prop:KRa} yields 
for all $t\in[T_1,T_2]$
\begin{multline*}
\frac{d}{dt}\|\eta_tv_t\|^{2}+\frac{1}{4}\Vert\vert\nabla (\eta_t v_t)\vert\Vert^2
\leq 
\frac{d}{dt}\|\eta_tv_t\|^{2}+\frac12 \Vert \eta_t\vert \nabla v\vert \Vert^2+\frac12\Vert \vert \nabla \eta_t\vert\Vert_X^2 \Vert\Eins_Bv\Vert^2
\\
\leq 
\Vert\Eins_Bv_t\Vert^2
\left(2\eta_t\frac{d}{d t}\eta_t+166\|\vert\nabla \eta_t\vert\|_{X}^{2}\right)+\frac12\Vert \vert \nabla \eta_t\vert\Vert_X^2 \Vert\Eins_Bv\Vert^2
\leq 
\Vert\Eins_Bv_t\Vert^2
\left(2\eta_t\frac{d}{d t}\eta_t+167\|\vert\nabla \eta_t\vert\|_{X}^{2}\right).
\end{multline*}
Since $F(0)=0$ and $F$ is convex, we get \[F(x)=\frac{F(x)-0}{x-0}\cdot  x\leq (\sup_{\RR} F' )\cdot x\] for all $x\geq 0$. Hence, we have \[v=F((u-\epsilon)_+)\leq (\sup_{\RR} F') \cdot  (u-\epsilon)_+\]
and thus the claim follows.
\end{proof}

In the following, for $0<t_1<t_2\leq T$, we will frequently use the distinguished time-cut-off function $\chi\colon\RR\to\RR$ given by 
\[
\chi(t)=
\begin{cases}
0 &\colon 0\leq t\leq t_1,\\
\frac{t-t_1}{t_2-t_1} &\colon t_1\leq t\leq t_2,\\
1 &\colon t_2\leq t\leq T.
\end{cases}
\]
Note that we have $\chi\leq 1$ and $
\tfrac{d}{dt}\chi
= (t_2-t_1)^{-1} \Eins_{ (t_1,t_2)}$ on $(0,T)\setminus\{t_1,t_2\}$. 
\begin{lemma}\label{lem:energyestimate}Let $B\subset X$ finite, $0<t_1<t_2\leq T$, $\chi$ as above,
$0\leq \phi\leq \Eins_{B^\circ}$, and $v$, $F$, $\epsilon$ as in Lemma~\ref{lem:ffunctions}. Then we have for any $t\in [t_2, T]$ 
\begin{multline*}
4
\Vert\eta_tF(u_t-\epsilon)\Vert^{2} +
\int_{t_1}^t
 \Vert \nabla(\eta_s F(u_s-\epsilon))\Vert^2\drm s
\leq 668(\sup F')\left(\frac{1}{t_2-t_1}+\Vert|\nabla\phi|\Vert_X^2\right) 
\int_{t_1}^t \|\Eins_{ B}(u_s-\epsilon)_+\|^{2}
\drm s.
\end{multline*}
\end{lemma}
\begin{proof}Set $\eta_t=\chi(t)\phi$.
Since $\phi\leq 1$ we obtain
\[
\frac{d}{dt}\eta_t=\frac1{t_2-t_1} \Eins_{ (t_1,t_2)}\phi\leq \frac1{t_2-t_1} \Eins_{ (t_1,t_2)}.
\]
From $\chi\leq 1$ we infer
\[
\||\nabla\eta_t|\|_X=\chi(t)\|\nabla\phi\|_X\leq \||\nabla\phi|\|_X,
\]
such that, using $\eta\leq 1$ ,
\[
2\eta \frac{d}{dt}\eta_t+167\Vert\vert\nabla \eta_t\vert\Vert_{X}^{2}
\leq 
167\left(\frac1{t_2-t_1} \Eins_{ (t_1,t_2)}+\||\nabla\phi|\|_X^{2}\right).
\]
Moreover, $\supp \eta_t=\supp\phi\subset B$.
Apply Lemma~\ref{lemma:maxandintgeneral}
to $v=F(u-\epsilon)$ and the above choice for $\eta$ to obtain for all 
$t\in(0,T]\setminus\{t_1,t_2\}$
\begin{multline*}
4\frac{d}{dt}\Vert\eta_tF(u_t-\epsilon)\Vert^{2}+
\Vert\vert\nabla (\eta_t F(u_t-\epsilon))\vert\Vert^2
\leq 
 4(\sup F') \|\Eins_{B}(u_t-\epsilon)_+\|^{2}
\left(2\eta_t\frac{d}{d t}\eta_t+167\Vert\vert\nabla \eta_t\vert\Vert_{X}^{2}\right)
\\
\leq 688(\sup F')\|\Eins_{B}(u_t-\epsilon)_+\|^{2}
\left(\frac{1}{t_2-t_1}\Eins_{ (t_1,t_2)}+\Vert|\nabla\phi|\Vert_X^2
\right).
\end{multline*}
If we let $t\in[t_2,T]$ and integrate the inequality above over $[t_1,t]$, we obtain the claim by using $\eta_{t_1}=0$.
\end{proof}
\subsection{De Giorgi iteration steps}\label{subsection:deGiorgielementary}

In this section we provide the elementary de Giorgi iteration steps, i.e., based on the maximal inequality for positive parts of subsolutions, we bound the norm on space-time cylinders of a subsolution in terms of the norm an a larger cylinder. 
To fix notation, 
let $o\in X$, $0<b_1<b_2<\infty$, $0< r_2<r_1\leq R$, $0<t_1<t_2\leq T$, $u$ a subsolution in $X\times [t_1,t_2]$, and 
\[ 
a_i=\int_{t_i}^{T}\|\Eins_{B_o(r_i)}(u_s-b_i)_+\|^2\ \drm t=\int_{t_i}^{T}\sum_{B_o(r_i)}m(u_s-b_i)_+^2\ \drm t, \quad i\in\{1,2\}.
\]

First, we show that on any graph with intrinsic metric, the elementary step always holds if the smaller ball is just one vertex.
\begin{lemma}[Elementary estimate]\label{lem:elementarypoint} If $r_1-r_2>S$, then we have
\[
m(o)(u_T(o)-b_2)_+^2
\leq 
2^8
\left(
\frac{1}{T-t_1}+\frac{1}{(r_1-r_2-S)^2}\right)a_{1}.
\]
\end{lemma}

\begin{proof}
Let $\{F_n\}_{n\in\NN}$ be a sequence of functions such that each $F_n$ satisfies the conditions of Lemma~\ref{lem:ffunctions}, 
$F_n\to (\cdot)_+$ uniformly, and $\sup_{n\in\NN,\RR}F_n'\leq 1$. Such $F_n$ exists by integrating twice a hat function $\phi\in\cC_c([0,1])$ with $\int_0^1\phi(t)\drm t=1$ and setting $F_n(t)=n\phi(t/n)$, $n\in\NN$, cf.~\cite{GrigoryanHH}.
\\
Since $b_1<b_2$ and hence $(u_t-b_2)_+\leq (u_t-b_1)_+$, we clearly have
\begin{align*}
m(o)(u_T(o)-b_2)_+^2
\leq 
\sum_{B(r_2)}m (u_T-b_1)_+^2
=\lim_{n\to\infty} \|\Eins_{B(r_2)} F_n(u_T-b_1)\|^2.
\end{align*}
Set $\phi=\phi_{B(r_2),r_1-S}$ and $\eta=\chi\phi$, where $\chi$ is given before Lemma~\ref{lem:energyestimate}. Since $\chi_T=1$, Proposition~\ref{prop:cutoff} yields $\||\nabla\eta_T|\|_X=\chi_T\||\nabla\phi|\|_X\leq (r_1-r_2-S)^{-1} $. Dropping the second summand on the left-hand side in Lemma~\ref{lem:energyestimate} we obtain for $\epsilon=b_1$, since $r_2<r_1$, $t_1< t=t_2=T$, $\chi_{t_1}=0$, $\chi_T=1$,
\[
\|\Eins_{B(r_2)} F_n(u_T-b_1)\|^2
\leq 
\|\eta_T F_n(u_T-b_1)\|^2
\leq 
2^8(\sup F_n')
\left(
\frac{1}{T-t_1}+\frac{1}{(r_1-r_2-S)^2}\right)
\int_{t_1}^T
\|\Eins_{B(r_1)}(u_t-b_1)_+\|^2\ \drm t.
\]
Note $\sup_{n\in\NN,\RR} F_n'\leq 1$ to obtain the claim.
\end{proof}

The following lemma is an elaboration of the results in \cite{GrigoryanHH}, the difference being the purely discrete setting with an intrinsic metric and the restriction to centered balls with large enough difference on the radii.

\begin{lemma}\label{lem:elementarystep}
Assume $r_1-r_2>4S$ and that there are constants $a,n>0$ such that for all $U\subset B_o(r_1) $, we have
\[
\lambda(U)\geq a\ m(U)^{-{\frac{2}{n}}}.
\]
Then we have 
\begin{align*}
a_2
\leq 
\frac{2^{{\frac{28}{n}}}}{a(b_2-b_1)^{{\frac{4}{n}}}}
 \left(\frac{1}{t_2-t_1}+\frac{1}{(r_1-r_2-4S)^2}\right)^{1+{\frac{2}{n}}} 
a_1^{1+{\frac{2}{n}}}.
\end{align*} 
\end{lemma}
\begin{proof}
Let $B_i:=B_o(r_i)$,
\[
r:=r_1-r_2, \quad \xi=b_1+\frac{1}{2}(b_2-b_1),
\]
and
\[ U:=B\left(r_2+\frac{r}{2}\right),
\quad
\tilde B:=B\left(r_2+\frac{3}{4}r\right).
\]
Then $b_1<\xi<b_2$ and $B_2\subset U\subset \tilde B\subset B_1$.
Let $\{F_n\}_{n\in\NN}$ be a sequence of functions such that each $F_n$ satisfies the conditions of Lemma~\ref{lem:ffunctions}, 
$F_n\to (\cdot)_+$ uniformly, and $\sup_{n\in\NN,\RR}F_n'\leq 1$. Such $F_n$ exists by integrating twice a non-negative hat function $\phi\in\cC_c([0,1])$ with integral one on the interval $[0,1]$ and setting $F_n(t)=n\phi(t/n)$, $n\in\NN$, cf.~\cite{GrigoryanHH}.
\\

Set $\phi=\phi_{B_2,r/2-S}$ and $\eta=\chi\phi$, where $\chi$ is given before Lemma~\ref{lem:energyestimate}. Since $r_1-r_2> 4S>2S$, Proposition~\ref{prop:cutoff} yields $\Vert|\nabla\phi|\Vert_X\leq  2(r_1-r_2-2S)^{-1}<\infty$. Lemma~\ref{lem:energyestimate} applied to $v=F_n(u-b_2)$, $n\geq 1$, $\epsilon=b_2$, $B=B_1$, $t=T$, and $\phi$ yields for all $n\geq 1$ by dropping the non-negative first summand on the left-hand side of the resulting inequality
\begin{multline*}
\int_{t_1}^T 
 \Vert \nabla(\eta_s F_n(u_s-b_2))\Vert^2\
\drm s
\leq 
 668\left(\frac{1}{t_2-t_1}+\frac{1}{(r_1-r_2-2S)^2}\right) 
\int_{t_1}^T \|\Eins_{ B_1}(u_s-b_2)_+\|^{2}
\ \drm s
\\
\leq 
 668\left(\frac{1}{t_2-t_1}+\frac{1}{(r_1-r_2-2S)^2}\right) 
\int_{t_1}^T \|\Eins_{ B_1}(u_s-b_1)_+\|^{2}
\ \drm s =a_1,
\end{multline*}
where we used that $b_1<b_2$ implies $(u-b_2)_+\leq (u-b_1)_+$.
\\

Now, for fixed $s\in (0,T]$, consider
\[
E_s:=U\cap \{u_s\geq b_2\}\subset U.
\]
By $\supp \phi\subset U$ and $F_n(t)=0$ if $t\leq 0$,  we get $\supp\phi F_n(u_s-b_2)\subset E_s$. Since $\phi\restriction_{B_2}=1$ and $\chi(s)=1$ for $s\in[t_2,T]$,  we have $\eta_s(x)=\chi(s)\phi(x)=1$ for all $x\in B_2$ and $s\in [t_2,T]$. Therefore, the definition of the smallest Dirichlet eigenvalue $\lambda(E_s)$ of $E_s$  implies for all $s\in [t_2,T]$
\[
\| \Eins_{B_2}F_n(u_s-b_2)\|^2
\leq \|\Eins_{B_1}\chi(s)\phi F_n(u_s-b_2)\|^2
\\\leq \|\Eins_{E_s}(\eta_s F_n(u_s-b_2))\|^2
\leq \frac{1}{\lambda(E_s)} \Vert \nabla(\eta_s F_n(u_s-b_2))\Vert^2.
\]
Integrating the resulting inequality yields
\[
\int_{t_2}^T \| \Eins_{B_2}F_n(u_s-b_2)\|^2 \ \drm s
\leq 
\int_{t_2}^T 
\frac{1}{\lambda(E_s)} \Vert \nabla(\eta_s F_n(u_s-b_2))\Vert^2
\ \drm s
\\
\leq 
\sup_{s\in[t_2,T]}\frac{1}{\lambda(E_s)} 
\int_{t_2}^T 
 \Vert \nabla(\eta_s F_n(u_s-b_2))\Vert^2
\ \drm s.
\]
Since $t_1<t_2$, we have
\[
\int_{t_2}^T 
 \Vert \nabla(\eta_s F_n(u_s-b_2))\Vert^2
\ \drm s\leq 
\int_{t_1}^T 
 \Vert \nabla(\eta_s F_n(u_s-b_2))\Vert^2
\ \drm s
\\
\leq 
668
 \left(\frac{1}{t_2-t_1}+\frac{1}{(r_1-r_2-2S)^2}\right) 
a_1.
\]
Hence, the uniform convergence of $(F_n)$ implies
\[
a_2
=\lim_{n\to\infty}
\int_{t_2}^T \sum_{B_2} m F_n(u_s-b_2)^2 \drm s
\leq 
668\sup_{s\in[t_2,T]}\frac{1}{\lambda(E_s)} 
 \left(\frac{1}{t_2-t_1}+\frac{1}{(r_1-r_2-2S)^2}\right) 
a_1.
\]

Next, we estimate the supremum involving the Dirichlet eigenvalues.
For any $s\in(0,T]$ we have $E_s\subset U\subset B(r_1)$.
Since $B_1$ satisfies the Faber-Krahn inequality by assumption, we get
\[
\sup_{s\in[t_2,T]}\frac{1}{\lambda(E_s)}\leq 
\sup_{s\in[t_2,T]}
\frac{1}{a}m(E_s)^{\frac{2}{n}}.
\]
In order to bound $m(E_s)$, use $\xi=b_2+(b_2-b_1)/2$, $u_s\geq b_2$ on $E_s$, and $E_s\subset U$ to get
\begin{multline*}
m(E_s)
=
\frac{4}{(b_2-b_1)^2}\|\Eins_{E_s} (b_2-\xi)_+\|^2 \leq 
\frac{4}{(b_2-b_1)^2}\|\Eins_{E_s} (u_s-\xi)_+\|^2 
\leq
\frac{4}{(b_2-b_1)^2}\|\Eins_U (u_s-\xi)_+\|^2
.
\end{multline*} 

In order to bound the quantity $\|\Eins_U (u_s-\xi)_+\|^2$, we need some preparations.
Set $\varphi=\phi_{U,r/4-S}$. Since $r_1-r_2>4S$ we have  $\|\vert\nabla\varphi\vert\|_X^2\leq (r/4-S)^{-2}= 16(r_1-r_2-4S)^{-2}<\infty$.
Apply Lemma~\ref{lem:energyestimate} to $v=F(u-\xi)$ with $F=F_n$, $n\geq 1$, $\epsilon=\xi$, and $\eta=\chi\varphi$ with $\chi$ given before Lemma~\ref{lem:energyestimate}. Drop the non-negative energy term to obtain for $t\in[t_2,T]$ 
\[
4
\Vert\eta_tF_n(u_t-\xi)\Vert^{2} 
\leq 668(\sup F_n')\left(\frac{1}{t_2-t_1}+\frac{16}{(r_1-r_2-4S)^{2}}\right) 
\int_{t_1}^t \|\Eins_{ U}(u_s-\xi)_+\|^{2}
\drm s.
\]
Since $t_1<t_2\leq T$, $U\subset B_1$, and $b_1<\xi$, we have 
$$\int_{t_1}^t\|\Eins_{U}(u_s-\xi)_+\|^2 \ \drm s \leq \int_{t_1}^T\|\Eins_{B_1}(u_s-b_1)_+\|^2\ \drm s=a_1.$$
By definition, we have $\chi(t)=1$, $F_n((u_t-\xi)_+)=F_n(u_t-\xi)$, and $\sup F_n'\leq 1$. Hence, since $2672\leq 2^{12}$, we obtain
\[
\Vert\varphi F_n((u_t-\xi)_+)\Vert^{2}
=
\Vert\eta_tF_n(u_t-\xi)\Vert^{2}
\leq 2^{12}
\left(\frac{ 1}{t_2-t_1}+\frac{1}{(r_1-r_2-4S)^2}\right)
a_1.
\]
Using $\varphi\restriction_U=1$ and $F_n\to F$ uniformly, we finally get
\begin{multline*}
\|\Eins_U(u_t-\xi)_+\|^{2}
\leq \Vert \varphi (u_t-\xi)_+\Vert^2
=\lim_{n\to\infty}
 \Vert \varphi F_n((u_t-\xi)_+)\Vert^2
 \leq 
 2^{12}
\left(\frac{ 1}{t_2-t_1}+\frac{1}{(r_1-r_2-4S)^2}\right)
a_1.
\end{multline*}
Plug the result into the bound for the measure of $E_s$ to obtain the result.
\end{proof}

%
%
%
%
%
%
%
\subsection{The iteration steps and mean value inequality}\label{subsection:iteration}
%
%
%
%
%
%
%
%
%

%
%

In this section we carry out Grigor'yan's iteration procedure to obtain an $\ell^2$-mean value inequality for non-negative subsolutions of the heat equation, \cite{Grigoryan-94,GrigoryanHH}. The difficulty of the argument lies in the fact that the iteration procedure naturally stops after a certain amount of steps depending on the existence of good cut-off functions. We resolve this by incorporating the mean-value inequality from Lemma~\ref{lem:elementarypoint}. The error produced by this procedure depends on the Faber-Krahn constant and the inverted vertex measure of the ball's center, but becomes small if the argument is run on large balls.

\begin{theorem}\label{thm:mv}
Let $x\in X$, $a,n>0$, $R\geq 288S$ such that for all $U\subset B_x(R)$ 
\[
\lambda(U)\geq a  \ m(U)^{-{\frac{2}{n}}}.
\]
Then there exists  a constant $C=C_{n,S}>0$ such that for all non-negative subsolutions $u$ on $(0,T]\times B_x(R)$
\[
u_T(x)^2
\leq 
C
\left[1\vee \frac{a^{\frac{n}{2}}}{m(x)}\right]^\theta
\frac{1}{a^{\frac{n}{2}}(T\wedge R^2)^{\frac{n}{2}+1}}
 \int_0^T \sum_{B_x(R)} m u_t^2\ \drm t,
\]
where 
\[
\theta(n,R)=\frac{n+2}{n}\cdot 1\wedge \left(\frac{288S}{R}\right)^{\frac{1}{n+2}}.
\]
\end{theorem}

\begin{proof}The proof is divided into three steps. First, since the Faber-Krahn inequality still holds if we shrink the radius with the same constants $a$ and $n$, we are able to iterate Lemma~\ref{lem:elementarystep}. This is possible as long as radii are large enough. Second, when the iteration procedure stops, we use Lemma~\ref{lem:elementarypoint} to mimick the iteration procedure to radius zero. The third step then consists in gluing together the resulting inequalities and choosing the cut-off level appropriately.
\\

We start with the first step, i.e., iterating Lemma~\ref{lem:elementarystep}.
Let $T_0=T/2$, $b\geq 0$, $\delta\in(0,1]$ to be chosen later,
\[
R_k=(1+2^{-k})\frac{R}{2},\qquad b_k=(1-2^{-k})b, \qquad T_{k+1}=T_k +\delta(R_k-R_{k+1})^2,\qquad \kappa=\left\lfloor \log_2\left(\frac{1}{8}\sqrt{\frac{R}{S}}\right)\right\rfloor,
\]
and
\[
a_k=\int_{T_k}^T\| \Eins_{B_x(R_k)}  (u_s-b_k)_+ \|^2\drm s. 
\]
Since $R\geq 64S$ we have for all $k\leq \kappa=\left\lfloor \log_2\left(\frac{1}{8}\sqrt{\frac{R}{S}}\right)\right\rfloor\leq \log_2\left(\frac{1}{8}\sqrt{\frac{R}{S}}\right)$
\[
R_k-R_{k+1}=\frac{R}{4}2^{-k}
\geq \frac{R}{4}2^{-{\kappa}}\geq \frac{R}{4}2^{-\log_2\left(\frac{1}{8}\sqrt{\frac{R}{S}}\right)}
=2\sqrt {RS}
\geq 8S>4S.
\] 
Since the Faber-Krahn inequality holds in $B_x(R)$, it also holds in all balls $B_x(r)$, $r\in[0,R]$ with the same constants $a$ and $n$. Thus, setting $q:={1}/({1+{\frac{2}{n}}})={n}/({n+2})$,
 Lemma~\ref{lem:elementarystep} yields for all $k\leq \kappa$
\[
a_{k+1}\leq C_k a_k^{\frac1q},
\]
where 
\[
C_k=
\frac{2^{{\frac{28}{n}}}}{a(b_{k+1}-b_k)^{{\frac{4}{n}}}}
 \left(\frac{1}{T_{k+1}-T_k}+\frac{1}{(R_k-R_{k+1}-4S)^2}\right)^{\frac{1}{q}}.
\]
Iterating the above inequality from $k=0$ to $\kappa$, we obtain
\[
a_{\kappa+1}^{q^{\kappa+1}}
\leq \left(\prod_{k=0}^{\kappa} C_k^{(\frac{1}{q})^{\kappa-k}}\right)^{q^{\kappa+1}}
a_0
=
\left(\prod_{k=0}^{\kappa} C_k^{q^{k+1}}\right)
a_0
.
\]
We estimate the product on the right-hand side and start with $C_k$. Observe that since 
\[
b_{k+1}-b_k=b2^{-k-1},
\]
we have for the first factor of $C_k$, using $4/n\leq 2(1+2/n)=2/q$,
\[
\frac{2^{{\frac{28}{n}}}}{a(b_{k+1}-b_k)^{{\frac{4}{n}}}}
=
\frac{2^{{\frac{32}{n}}}2^{{\frac{4}{n}} k}}{ab^{{\frac{4}{n}}}}
\leq 
\frac{2^{{\frac{32}{n}}}2^{\frac{2}q k}}{ab^{{\frac{4}{n}}}}.
\]
Next, we bound the second factor. By the definition of $T_k$, we have
\[
T_{k+1}-T_k=\delta \frac{R^2}{16}4^{-k}.\]
Since $R\geq 16S$, we have $\sqrt{R/S}\leq R/(4S)$ and hence $\kappa=\log_2\left\lfloor\sqrt{{R}/{S}}/8\right\rfloor \leq \log_2\left\lfloor{R}/(32S)\right\rfloor$. This yields $16S/R\leq 2^{-(\kappa+1)}\leq 2^{-(k+1)}$ for all $k\leq \kappa$.
Thus, by the definition of $R_k$ and $\delta\leq 1$, we obtain
\[ 
R_k-R_{k+1}-4S =\frac{R}{4}\left(2^{-k}-\frac{16S}{R}\right)
\geq 
\frac{R}{4}(2^{-k}-2^{-(k+1)})
=
\frac{R}{8}2^{-k}
\geq \sqrt\delta \frac{R}{8} 4^{-k}.
\]

Hence, for the second factor of $C_k$, we get
\[
 \left(\frac{1}{T_{k+1}-T_k}+\frac{1}{(R_k-R_{k+1}-4S)^2}\right)^{\frac{1}{q}}
\\ \leq 
\left(\frac{16 \cdot 4^k}{\delta R^2}+\frac{64\cdot 16^k}{\delta R^2}\right)^{\frac{1}{q}}
\leq \left(\frac{2^{7}2^{4k}}{\delta R^{2}}\right)^{\frac{1}{q}}.
\]
Plugging in the above estimates for the factors of $C_k$ we obtain for the product
\[
\prod_{k=0}^{\kappa} C_k^{q^{k+1}}
\leq
\prod_{k=0}^{\kappa} \left(\frac{2^{{\frac{32}{n}}}2^{\frac{2}q k}}{ab^{{\frac{4}{n}}}}\left(\frac{2^{7}2^{4k}}{\delta R^{2}}\right)^{\frac1q}\right)^{{q^{k+1}}}
=
2^{6\sum_{k=0}^{\kappa}kq^{k}}
 \left(\frac{2^{7+{\frac{46}{n}}}}{ab^{{\frac{4}{n}}}(\delta R^2)^{\frac1q} }\right)^{\sum_{k=0}^{\kappa}q^{k+1}}.
\]
We have 
\[
\sum_{k=0}^{\kappa}q^{k+1}=q\frac{1-q^{\kappa+1}}{1-q}
=\frac{n}{2}(1-q^{\kappa+1}),
\qquad 
\sum_{k=0}^{\kappa}kq^{k}
=q\frac{(\kappa-1)q^{\kappa}-\kappa q^{\kappa-1}+1}{(1-q)^2}
=\frac{n^2}{4}\frac{1-({\frac{2}{n}} \kappa+1)q^{\kappa}}{q}
\leq \frac{n^2}{4}\frac{1}{q}.
\]
Together with $\frac{n}{2}(1-q^{\kappa+1})\leq \frac{n}{2}$ we obtain
\[
2^{6\sum_{k=0}^{\kappa}kq^{k}}\left(2^{7+{\frac{46}{n}}}\right)^{\sum_{k=0}^{\kappa}q^{k+1}}
\leq 2^{6\frac{n^2}{4}\frac{1}{q}}
\left(2^{7+{\frac{46}{n}}}\right)^{\frac{n}{2}}\leq 2^{n^2+5n+25}=:C_n'.
\]
Choose 
\[
\delta=\frac{T\wedge R^2}{R^2} \in (0,1]
\]
and use $nq/2=n/2+1$ to obtain
\begin{align*}
\prod_{k=0}^{\kappa} C_k^{q^{k+1}}
\leq 
C_{n}'
\left( \frac{1}{a^{\frac{n}{2}}b^{2}(\delta R^2)^{\frac{n}{2}+1}}\right)^{1-q^{\kappa+1}}
=
C_{n}'
\left(\frac{1}{
a^{\frac{n}{2}}b^2(T\wedge R^2)^{\frac{n}{2}+1}}\right)^{1-q^{\kappa+1}}
.
\end{align*}

Next, we provide the second step and bound $a_{\kappa+1}$ below by $(u_T-b)_+$. \\
Apply Lemma~\ref{lem:elementarypoint} to $b_{\kappa+1}<b$, $r_2=R/2<R_{\kappa+1}=r_1$, $t_1=T_{\kappa+1}<T$, and 
 $r_1-r_2=R_{\kappa+1}-R/2>S$, to obtain
\[
m(x)(u_T(x)-b)_+^2
\leq 
2^8
\left(
\frac{1}{T-T_{\kappa+1}}+\frac{1}{(R_{\kappa+1}-R/2-S)^2}\right)a_{\kappa+1}.
\]
In order to bound the term in parentheses from above, we estimate $T_{\kappa+1}$ and $R_{\kappa+1}-R/2-S$ and start with
$T_{\kappa+1}$. Iterating the identity
\[
T_{k}
=T_{k-1}+\delta\frac{R^2}{4}4^{-k}
\] for $k\in\NN$ and using $T_0=T/2$, $\delta=(T\wedge R^2)/R^2$, we obtain
\[
T_{k}
=T_0+\delta \frac{R^2}{4}\sum_{l=0}^{k-1}\left(\frac{1}{4}\right)^l
==T_0+\delta \frac{R^2}{4}\frac{1-4^{-(k-1)}}{1-4^{-1}}
=\frac{T}{2}+\delta \frac{R^2}{3}(1-4^{-(k-1)})
=
\frac{T}{2}+ \frac{T\wedge R^2}3(1-4^{-(k-1)}).
\]
Thus, we have $T_k\leq 2T/3$, $k\in\NN_0$, in particular $T_{\kappa+1}\leq 2T/3$.
\\
Next, we bound $R_{\kappa+1}-R/2-S$ from below.
Since $R\geq 288S$ we have $R/S\geq 1$ and thus $\sqrt{R/S}\leq R/S $. This yields $\kappa\leq \log_2(\sqrt{{R}{S}}/8)\leq\log_2\left({R}/({8S})\right)$, and therefore
$2^{-(\kappa+1)}{R}/{4}\geq S$.
Hence,
\[
R_{\kappa+1}-\frac{R}{2}-S
=
2^{-(\kappa+1)}\frac{R}{2}-S
\geq{2^{-(\kappa+1)}}\frac{R}{4}= {2^{-(\kappa+3)}}{R}.
\]
Together with the bound $T_{\kappa+1}\leq 2T/3$ from above and $3+2^{\kappa+3}\leq 2^{\kappa+4}$, we obtain
\[
(u_T(x)-b)_+^2
\leq 
\frac{2^8}{m(x)}
\left(
\frac{3}{T}+\frac{2^{\kappa+3}}{R^2}\right)a_{\kappa+1}\leq 
\frac{2^{\kappa+12}}{m(x)( T\wedge R^2)}a_{\kappa+1}.
\]

In the final third step, we put the relations above together. 
Raise the last inequality to the power $q^{\kappa+1}$, multiply by $b^{2(1-q^{\kappa+1})}$, and use $a_{\kappa+1}^{q^{\kappa+1}}\leq C_{n}'
\left( a^{\frac{n}{2}}\rho^2(T\wedge R^2)^{\frac{n}{2}+1}\right)^{-(1-{q^{\kappa+1}})} a_0$ from our first step to get
\begin{multline*}
b^{2(1-{q^{\kappa+1}})}(u_T(x)-b)_+^{2{q^{\kappa+1}}}
\leq 
b^{2(1-{q^{\kappa+1}})}\left(\frac{2^{\kappa+12}}{m(x) (T\wedge R^2)}
\right)^{{q^{\kappa+1}}} 
C_{n}'
\left( \frac{1}{a^{\frac{n}{2}}b^2(T\wedge R^2)^{\frac{n}{2}+1}}\right)^{1-{q^{\kappa+1}}} a_0
\\
= C_{n}'
2^{(\kappa+12){q^{\kappa+1}}}(T\wedge R^2)^{\frac{n}{2}{q^{\kappa+1}}}\left[\frac{a^{\frac{n}{2}}}{m(x)}\right]^{q^{\kappa+1}}
\frac{
1}{a^{\frac{n}{2}}(T\wedge R^2)^{\frac{n}{2}+1}} a_0.
\end{multline*}
Choose $b=u_T(x)/2$ to obtain
\[
u_T(x)^2
\leq 
4 C_{n}'2^{(\kappa+12){q^{\kappa+1}}} (T\wedge R^2)^{\frac{n}{2}{q^{\kappa+1}}}
\left[\frac{a^{\frac{n}{2}}}{m(x)}\right]^{q^{\kappa+1}}
\frac{
1}{a^{\frac{n}{2}}(T\wedge R^2)^{\frac{n}{2}+1}}
a_0.
\]

Finally, we estimate the terms involving ${q^{\kappa+1}}$ in the upper bound. To this end, observe 
\[
f(\kappa)=\kappa{q^{\kappa+1}} 
\]
attains its maximium at $\kappa_0=1/(\ln(1/q))$ with 
$f(\kappa_0)=1/[\euler(1+{\frac{2}{n}})\ln(1+{\frac{2}{n}})]^{-1}\leq 1/\ln(1+{\frac{2}{n}})\leq n+2$, where we used $\euler(1+{\frac{2}{n}})\geq 1$ and the mean value theorem in the last estimate. Since $q\leq 1$ we have $q^{\kappa+1}\leq 1$. Hence,
\[
(\kappa+12){q^{\kappa+1}} \leq n+2+12=n+14.
\]
Estimate the power of two in the estimate for $u_T$ above by this quantity. Now we turn to more precise estimates on $q^{\kappa+1}$. 
Since 
$\kappa= \left\lfloor\log_2\sqrt{{R}/{64S}}\right\rfloor\geq\log_2\sqrt{{R}/{64S}}-1$ and $q\leq 1$, we have 
\[
q^{\kappa+1}\leq 
q^{\log_2\sqrt{{R}/{64S}}}
=\left(\frac{64S}{R}\right)^{\log_2\sqrt{1/q}}
\leq \theta(n,R),
\]
where we use in the last step the same reasoning as in Theorem~\ref{thm:asympdoubling}.
Furthermore, by defining the function $f(x)=x^{\left(1/x\right)^{\log_2\sqrt{1/q}}} $, we obtain, using $n/q=n(n+2)/n=n+2$,
\[
(T\wedge R^2)^{\frac{n}{2}q^{\kappa+1}}
\leq 
1\vee R^{n\left(\frac{64S}{R}\right)^{\log_2\sqrt{1/q}}}
=
1\vee (64S)^{n\left(\frac{64S}{R}\right)^{\log_2\sqrt{1/q}}}f(R/64S)^n\leq (1\vee S)^n2^{n^2+2n},
\]
where we used that $R\geq S$ yields  $(64S)^{n\left({S}/{R}\right)^{\log_2\sqrt{1/q}}}\leq 2^{6n}(1\vee S)^{n}$ and $f\leq 2^{n+2}$, cf.~proof of Theorem~\ref{thm:asympdoubling}. Applying these estimates and setting $C_{n,S}=(1\vee S)^n2^{2n^2+27n+51}$ yields the claim.
\end{proof}

\section{Integral maximum principle}\label{section:intmaxprinciple}

In this section, we prove the integrated maximum principle for locally finite graphs supporting an intrinsic metric with finite jump size. First results in this direction go back to \cite{Delmotte-99, Folz-11, BauerHuaYau-15}, transferring Grigor'yan's original result from manifolds to graphs using Hamilton-Jacobi inequalities related to manifolds. However, these results are not optimal and do not incorporate the bottom of the spectrum of the Laplacian. The Davies-Gaffney-Grigor'yan lemma was then proved in  \cite{BauerHuaYau-17}. By a refined analysis of these results, we provide integral maximum principles for all functions satisfying an adapted Hamilton-Jacobi equation. 
\\

Define
\[
\mathrm d\Gamma(\phi,\psi)(x)
:=
\sum_{y\in X}\frac{b(x,y)}{m(x)}\vert \nabla_{xy}\phi\nabla_{xy}\psi\vert, \quad x\in X.
\]

The proofs of the next results we will make use of the following elementary lemma.
\begin{lemma}\label{lem:elementary}Let $u,\omega\in\cC(X)$ and $x,y\in X$. The following hold.
\begin{enumerate}[(i)]
\item 
$
\nabla_{xy}u\nabla_{xy}(u\euler^{\omega})
=\vert \nabla_{xy}(u\euler^{\frac{\omega}2})\vert^2-u(x)u(y)\vert\nabla_{xy}\euler^{\frac{\omega}2}\vert^2.$
\item 
$
\vert \nabla_{xy}\euler^{\frac{\omega}2}\nabla_{xy}\euler^{-\frac{\omega}2}\vert=2\left(\cosh\left(\frac{\omega(x)-\omega(y)}2\right)-1\right)
=
\euler^{-\frac{1}{2}(\omega(x)+\omega(y))}\vert\nabla_{xy}\euler^{\frac{\omega}2}\vert^2.
$
\item 
$
u(x)u(y)\vert\nabla_{xy}\euler^{\frac{\omega}2}\vert^2
\leq 
\frac{1}{2}u(x)^2\euler^{\omega(x)}\vert \nabla_{xy}\euler^{\frac{\omega}2}\nabla_{xy}\euler^{-\frac{\omega}2}\vert
+
\frac{1}{2}u(y)^2\euler^{\omega(y)}\vert \nabla_{xy}\euler^{\frac{\omega}2}\nabla_{xy}\euler^{-\frac{\omega}2}\vert.
$
\end{enumerate}
\end{lemma}
\begin{proof}
Items (i) and (ii) follow by direct computation. For (iii), use (ii) to get 
\[u(x)u(y)\vert\nabla_{xy}\euler^{\omega/2}\vert^2
=
u(x)\euler^{\frac{1}{2}\omega(x)}u(y)\euler^{\frac{1}{2}\omega(y)}\euler^{-\frac{1}{2}(\omega(x)+\omega(y))}\vert\nabla_{xy}\euler^{\omega/2}\vert^2
\\
=
u(x)\euler^{\frac{1}{2}\omega(x)}u(y)\euler^{\frac{1}{2}\omega(y)}\vert \nabla_{xy}\euler^{\frac{\omega}2}\nabla_{xy}\euler^{-\frac{\omega}2}\vert
\]
and apply $ab\leq \frac{1}{2}(a^2+b^2)$, $a,b\in\RR$.
\end{proof}

First, we show the integral maximum principle on finite subsets. A standard exhaustion argument yields the integral maximum principle on the whole graph. 
\\
To this end, we consider the heat equation subject to Dirichlet boundary conditions on subsets of a graph.
Let $I\subset \RR$, $u\colon I\times X\to \RR$ be continuous, and $B\subset X$. We say that $u$ solves the Dirichlet problem on $I\times B$ if 
\[
\begin{cases}
\frac{d}{dt} u= -\Delta^B u &\text{on} \ I^\circ\times B,\\
u(0,\cdot)=u\restriction_B (0,\cdot),&\\
u=0 & \text{on} \  I\times (X\setminus B),
\end{cases}
\]
where $\Delta^B$ denotes the Dirichlet Laplacian associated to $B$. Solutions to the above problem have the following property.

\begin{lemma}Let $I\subset\RR$ be an interval, $B\subset X$ finite, and let $u\geq 0$ solve the Dirichlet problem on $I\times B$. Further, let $\omega\colon I\times B \to\RR$ be continuous such that $t\mapsto \omega_t(x):=\omega(t,x)$ is continuously differentiable on $I^\circ$ for all $x\in B$  and such that 
\[
\mathrm d\Gamma(\euler^{\frac{\omega}2},\euler^{-\frac{\omega}2})\leq -\frac{d}{dt}\omega \quad\text{on} \ I^\circ\times B.
\]
Then the function
\[
\Xi\colon I\to [0,\infty), \quad t\mapsto  \euler^{2\lambda(B)t}\Vert u_t\euler^{\omega_t/2}\Vert^2
\]
is non-increasing. 
\end{lemma}

\begin{proof}
Set
\[\xi(t)= \Vert u_t\euler^{\omega_t/2}\Vert^2, \quad t\in I.
\]
First, we show $\xi'\leq -\lambda(B)\xi$. Then we use $\Xi(t)=\euler^{2\lambda(B)t}\xi(t)$, $t\in I$, to conclude.
\\
Since $u$ solves the Dirichlet problem on $B$, we obtain
\[
\xi'(t)=\frac{d}{dt}\sum_X m \ u^2_t\euler^{\omega_t}
= 2\sum_X m \ u_t\frac{d}{dt}u_t\euler^{\omega_t}+\sum_X m \ u_t^2\euler^{\omega_t}\frac{d}{dt} \omega_t
=
-2\sum_X m \ u_t\Delta u_t\euler^{\omega_t}+
\sum_X m \ u_t^2\euler^{\omega_t}\frac{d}{dt} \omega_t.
\]
Now, we show that the first sum can be bounded by the negative of the second sum, what leads to the claim. To this end, apply Green's formula and 
Lemma~\ref{lem:elementary} (i)
to the first sum to obtain
\begin{multline*}
-2\sum_X m \ u_t\Delta u_t\euler^{\omega_t}
=-\sum_{x,y\in X} b(x,y) \nabla_{xy}u_t\nabla_{xy}(u_t\euler^{\omega_t})
\\
=-\sum_{x,y\in X} b(x,y) \vert \nabla_{xy}(u_t\euler^{\omega_t/2})\vert^2
+
\sum_{x,y\in X} b(x,y)u_t(x)u_t(y)\vert\nabla_{xy}\euler^{\omega_t/2}\vert^2.
\end{multline*}
In order to treat the first summand, note that the first  Dirichlet eigenvalue of $B$ is given by the infimum of the Rayleigh quotient. Hence, we obtain
\[
-\sum_{x,y\in X} b(x,y) \vert \nabla_{xy}(u_t\euler^{\omega_t/2})\vert^2
\leq -2\lambda(B)\Vert u_t\euler^{\omega_t/2}\Vert_2^2=-2\lambda(B)\xi(t).
\] 
Note that the 2 appears due to Green's formula.
\\
Next, we turn to the second sum. 
Use Lemma~\ref{lem:elementary} (iii) 
to obtain by
multiplying by $b(x,y)$ and summing up 
\begin{multline*}
\sum_{x,y\in X} b(x,y)u_t(x)u_t(y)\vert\nabla_{xy}\euler^{\omega_t/2}\vert^2
\\
\leq 
\frac{1}{2}\sum_{x,y\in X} b(x,y)u_t(x)^2\euler^{\omega_t(x)}\vert \nabla_{xy}\euler^{\frac{\omega}2}\nabla_{xy}\euler^{-\frac{\omega}2}\vert
+
\frac{1}{2}\sum_{x,y\in X} b(x,y)u_t(y)^2\euler^{\omega_t(y)}\vert \nabla_{xy}\euler^{\frac{\omega}2}\nabla_{xy}\euler^{-\frac{\omega}2}\vert
\\
=
\sum_{X} m u_t^2\euler^{\omega_t}
\drm \Gamma(\euler^{\frac{\omega}2},\euler^{-\frac{\omega}2})\leq -\sum_X m \ u_t^2\euler^{\omega_t}\frac{d}{dt}\omega_t,
\end{multline*}
where we used the eikonal inequality for $\omega$ and the non-negativity of $u$. Hence, we obtain the inequality $\xi'(t)\leq -2\lambda(B) \xi(t)$, $t\in I^\circ$. 
\\
Since
$\Xi(t)=\euler^{2\lambda(B)t}\xi(t)$, $t\in I$, the above estimate on $\xi'$ yields
\[
\Xi'(t)=\euler^{2\lambda(B)t}(2\lambda(B)\xi(t)+\xi'(t))\leq 0,
\]
whence the claim.
\end{proof}

A simple exhaustion argument yields the following integral maximum principle, cf.~\cite[Section~3]{BauerHuaYau-15}.

\begin{theorem}[Integral maximum principle]\label{thm:intmaxprinciple}Assume that $u\geq 0$ solves the heat equation on $[0,\infty)\times X$  with $u(0,\cdot)\in\ell^2(X,m)$.
Let $\omega\colon [0,\infty)\times X\to\RR$ continuous such that $t\mapsto \omega_t(x):=\omega(t,x)$ is continuously differentiable on $(0,\infty)$ with continuous extension at zero for all $x\in X$ and such that  
\[
\mathrm d\Gamma(\euler^{\frac\omega2},\euler^{-\frac\omega2})\leq -\frac{d}{dt}\omega \quad\text{on}\ (0,\infty)\times X.
\]
Then the function
\[
\Xi\colon I\to [0,\infty), \quad t\mapsto  \euler^{2\Lambda t}\Vert u_t\euler^{\omega_t/2}\Vert_2^2
\]
is non-increasing. 
\end{theorem}

\section{Weighted $\ell^2$-norm of the heat kernel}\label{section:weightedl2norm}
In this paragraph, we bound certain weighted $\ell^2$-norms of the heat kernel provided Faber-Krahn inequalities hold in a ball. The original idea 
goes back to Grigor'yan on manifolds, cf.~\cite{Grigoryan-09}. However, due to the structure of the solutions of the graph version of the Hamilton-Jacobi equation in Section~\ref{section:intmaxprinciple},  the weights originally appear on manifolds do not lead to the sharp Gaussian expected by Davies' and Pang's results. The situation however is not completely hopeless since the Davies-Gaffney-Grigor'yan lemma was proven for graphs in \cite{BauerHuaYau-17}. Here, we will modify the weighted $\ell^2$-norm of the heat kernel such that they fit our new version of the Hamilton-Jacobi equation.

Another obstacle which appears is that the Faber-Krahn inequality yields bounds on weighted $\ell^2$-norms only for  large times. The argument for obtaining heat kernel bounds however requires bounds on these norms for all times. We solve this issue by incorporating a bound on weighted $\ell^2$-norms which always hold for all times on graphs. This can be found in Theorem~\ref{thm:integratedheat}, the final result of this section. 
\\

For $\eta>0$, we let
\[
h(\eta)=S^{-2}\big(\cosh(\eta S)-1\big),
\]
where $S$ is the jump size of the intrinsic metric $\rho$, 
and moreover for  $x\in X$, $t>0$, we define the weighted $\ell^2$-norm of the heat kernel by
\[
E_\eta(x,t)=\sum_{z\in X}m(z)
p_t(x,z)^2
\euler^{
2\eta\rho(x,z)-2th(\eta)}.
\]

Our estimates on these norms are based on the following observations, which are elaborations on corresponding results in \cite{Grigoryan-09}.
\begin{lemma}\label{lem:eikonalheatkernelpreliminary}
Let  $R,T,\nu\geq 0$, $o\in X$, and $f\colon X\to (0,\infty)$. 
Assume that for all non-negative subsolutions $u$ on $[0,T]\times B_o(R)$ we have
\[
u_T(o)^2\leq \nu\sum_X m \ u_0^2f.
\]
Then we have
\[
\sum_X m \ p_T(o,\cdot)^2f^{-1}\leq \nu.
\]
\end{lemma}
\begin{proof}
Let $\phi$ be a cut-off function and 
\[
u_0=p_T(o,\cdot)f^{-1}\phi.
\]
Then the assumption applied to the heat equation with initial datum $u_0$ yields
\[
\left(\sum_{X}m \ p_T(o,\cdot)^2f^{-1}\phi\right)^2
=u_T(o)^2
\\
\leq 
\nu
 \sum_{X} \ m\ u_0^2f
=
\nu
 \sum_{X} \ m\ p_T(o,\cdot)^2f^{-1}\phi^2.
\]
Since $\phi$ is an arbitrary cut-off function and hence $\phi^2\leq \phi$, cancelling the sum on both sides yields the claim.
\end{proof}

On graphs with an intrinsic metric, we have a pointwise mean-value inequality which holds for all times and vertices. A weighted $\ell^2$-norm estimate follows  directly from the integrated maximum principle. If there is an $\ell^2$-mean value inequality for subsolutions of the heat equation, a slightly different argument leads to a different estimate.
\begin{lemma}\label{lem:eikonalheatkernel}
Let  $R,T\geq 0$, $o\in X$, and
$\omega\colon [0,T]\times X\to\RR$ be continuous such that $t\mapsto \omega_t(x):=\omega(t,x)$ is continuously differentiable on $(0,T)$ with continuous extension on $[0,T]$, $x\in X$, $\omega_t\restriction_{B_x(R)}\geq 0$
for any $t\in[0,T]$, 
and 
\[
\mathrm d\Gamma(\euler^{\frac{\omega}2},\euler^{-\frac{\omega}2})\leq -\frac{d}{dt}\omega \quad\text{on}\ (0,T)\times X.
\] 
\begin{enumerate}[(i)]
\item We have
\[
\sum_{X}m \ p_T(o,\cdot)^2\euler^{-\omega_0}
\leq 
\frac{1}{m(o)}
 .
\]
\item If for all non-negative subsolutions $u$ on $[0,T]\times B_o(R)$ we have
\[
u_T(o)^2
\leq 
\sigma
\int_{0}^T \sum_{B_o(R)} \ m\ u_t^2 \ \drm t, 
\]
then
\[
\sum_{X}m \ p_T(o,\cdot)^2\euler^{-\omega_0}
\leq 
\sigma T.
\]
\end{enumerate}
\end{lemma}
\begin{proof}\begin{enumerate}[(i)]
\item
Let $u\geq 0$ be a solution on the cylinder $[0,T]\times X$. Since $\omega_t\restriction_{B_o(R)}\geq 0$ for any $t\in[0,T]$ and $\omega$ satisfies the eikonal inequality, the integral maximum principle in Theorem~\ref{thm:intmaxprinciple} yields 
\begin{align*}
m(o)u_T(o)^2
\leq 
 \sum_{B_o(R)} \ m\ u_T^2 
\leq 
\sum_{B_o(R)} \ m\ u_T^2\euler^{\omega_T} 
\leq 
\sum_{X} \ m\ u_0^2\euler^{\omega_0}.
\end{align*}
Hence, Lemma~\ref{lem:eikonalheatkernelpreliminary} applied with $\nu=1/m(o)$ and $f=\euler^{\omega_0}$ yields the claim.
\item By the same reasoning, we have 
\begin{align*}
u_T(o)^2
\leq 
\sigma
\int_{0}^T \sum_{B_o(R)} \ m\ u_t^2 \ \drm t
\leq 
\sigma \int_{0}^T \sum_{B_o(R)} \ m\ u_t^2\euler^{\omega_t} \ \drm t
\leq 
\sigma T \sum_{X} \ m\ u_0^2\euler^{\omega_0}.
\end{align*}
Hence, Lemma~\ref{lem:eikonalheatkernelpreliminary} applied with $\nu=\sigma T$ and $f=\euler^{\omega_0}$ yields the claim.\qedhere
\end{enumerate}
\end{proof}

Lemma~\ref{lem:eikonalheatkernel} will be applied to different choices of solutions of the Hamilton-Jacobi equation. Specifically, 
for  $o\in X$ and $R\geq 0$ let
\[
\rho_0:=(\rho(o,\cdot)-R)_+
\]
and for $T\geq0$ introduce the following functions:
\begin{align*}
\omega&\colon [0,T]\times X\to [0,\infty),\qquad (t,x)\longmapsto\omega_t(x):=-2\eta\rho_0(x)-2(t-T)h(\eta),
\\[1ex]
\tilde\omega&\colon [0,T]\times X\to [0,\infty),\qquad (t,x)\longmapsto\tilde\omega_t(x):=2\eta\rho_0(x)-2 t h(\eta),
\\[1ex]
\bar\omega&\colon [0,T]\times X \to [0,\infty),\qquad (t,x)\longmapsto\bar\omega_t(x):=2\eta\rho(o,x)-2th(\eta).
\end{align*}
\begin{lemma}\label{lem:omegaeikonal}The functions $\omega$, $\tilde\omega$, $\bar \omega$ satisfy the eikonal inequality.
\end{lemma}
\begin{proof}First, we consider $\omega$. 
For the first summand of the eikonal inequality, we obtain
\[
\frac{d}{dt}\omega=-2h(\eta).
\]
The second summand of the eikonal inequality is given by
\[
\drm  \Gamma(\euler^{\frac{\omega_t}2},\euler^{-\frac{\omega_t}2})(x)
=2\sum_{y\in X}\frac{b(x,y)}{m(x)}\left(\cosh\left(\frac{\omega_t(x)-\omega_t(y)}2\right)-1\right)
\\
=2\sum_{y\in X}\frac{b(x,y)}{m(x)}\left(\cosh(\eta(\rho_0(y)-\rho_0(x)))-1\right).\]
In order to bound the sum above, we look at non-trivial summand separately. Let $x,y$ with $b(x,y)>0$. Since 
\[
(0,\infty)\to(0,\infty), \quad t\mapsto
\frac{1}{t^2}(\cosh (\eta t)-1)
\]
is non-decreasing, we obtain by definition of $h(\eta)=S^{-2}\big(\cosh(\eta S)-1\big)$, the $1$-Lipschitz continuity of $\rho_0$ with respect to $\rho$, and $\rho(x,y)\leq S$ 
\[
\cosh( \eta(\rho_0(y)-\rho_0(x)))-1\leq \cosh(\eta\rho(x,y))-1
\leq\frac{\rho(x,y)^2}{S^2} \left(\cosh(\eta S )-1\right)
=\rho(x,y)^2h(\eta).
\]
Hence, since $\rho$ is intrinsic, we get
\[
\drm  \Gamma(\euler^{\frac{\omega_t}2},\euler^{-\frac{\omega_t}2})(x)
\leq 
2\sum_{y\in X}\frac{b(x,y)}{m(x)}\rho(x,y)^2h(\eta)
\leq 2h(\eta)\]
and hence
\[
\frac{d}{dt}\omega_t(x)+
\drm  \Gamma(\euler^{\frac{\omega_t}2},\euler^{-\frac{\omega_t}2})(x) \leq 0,
\]
what is the claim for $\omega$. Replacing $\omega$ by $\tilde\omega$ or $\bar\omega$, all the intermediate steps above hold as well and the claim for this function follows immediately.
\end{proof}

With the first of the three solutions of the Hamilton-Jacobi equation, we can prove the following.

\begin{lemma}\label{lem:kappabound}Let $T,R,\eta,\sigma\geq 0$, and
$o\in X$.
\begin{enumerate}[(i)]
\item We have 
\[
\sum_{X}m \ p_T(o,\cdot)^2\euler^{2\eta\rho_0-2Th(\eta)}
\leq 
\frac{1}{m(o)}
 .
\]
\item
If all non-negative subsolutions $u$ on $[0,T]\times B_o(R)$ 
\[
u_T(o)^2
\leq \sigma
\int_{0}^T \sum_{B_o(R)} \ m\ u_t^2 \ \drm t,
\]
then we have 
\[
\sum_{X}m \ p_T(o,\cdot)^2\euler^{2\eta\rho_0-2Th(\eta)}
\leq 
\sigma T
 .
\]
\end{enumerate}
\end{lemma}

\begin{proof}
The function $\omega$ satisfies the eikonal inequality by Lemma~\ref{lem:omegaeikonal} and $\omega\restriction_{B_o(R)}\geq 0$, whence (i) follows from Lemma~\ref{lem:eikonalheatkernel}(i) applied to $\omega$. Analogously, (ii) follows from Lemma~\ref{lem:eikonalheatkernel}(ii) .
\end{proof}

\begin{lemma}\label{lem:FKintheat}
Let $a,n>0$ and $R\geq 288S$ such that for all $U\subset B_o(R)$ 
\[
\lambda(U)\geq a \ m (U)^{-{\frac{2}{n}}}.
\]
Then there is a constant $C=C_{n,S}>0$ such that for all $T,\eta>0$ 
\[
\sum_{X}m \ p_T(o,\cdot)^2\euler^{2\eta\rho_0-2Th(\eta)}
\leq C
\left[1\vee \frac{a^{\frac{n}{2}}}{m(o)}\right]^\theta
\frac{a^{-\frac{n}{2}}}{(T\wedge R^2)^{\frac{n}{2}}},
\]
where  
\[
\theta=\theta(n,R)=\frac{n+2}{n}\cdot 1\wedge \left(\frac{288S}{R}\right)^{\frac{1}{n+2}}.
\]
\end{lemma}
\begin{proof}
Since $\tilde\omega_T=2\eta\rho_0-2Th(\eta)$ satisfies the eikonal inequality by Lemma~\ref{lem:omegaeikonal}, the left-hand side is non-increasing in $T$. Thus, it suffices to consider the case $T\leq R^2$.
By Theorem~\ref{thm:mv} we obtain for all subsolutions $u\geq 0$
\[
u_T(o)^2
\leq 
\sigma \int_0^T \sum_{B_o(R)} m u_t^2\ \drm t,
\]
where $
\sigma=C_{n,S}
\left[1\vee\frac{a^{\frac{n}{2}}}{m(o)}\right]^\theta
\frac{1}{a^{\frac{n}{2}}T^{\frac{n}{2}+1}}
$, $C_{n,S}>0$.
Thus, Lemma~\ref{lem:kappabound}(ii) yields
\[
\sum_{X}m \ p_T(o,\cdot)^2\euler^{2\eta\rho_0-2Th(\eta)}
\leq 
\sigma T
=C_n
\left[1\vee \frac{a^{\frac{n}{2}}}{m(o)}\right]^\theta
\frac{1}{a^{\frac{n}{2}}T^{\frac{n}{2}}}
\]
and whence the claim.
\end{proof}

The above preparations lead to the following estimate on the weighted $\ell^2$-norms of the heat kernel.
\begin{theorem}\label{thm:integratedheat}
Let $a,n>0$ be constants, $R\geq 288S$ such that for all $U\subset B_o(R)$ 
\[
\lambda(U)\geq a \ m (U)^{-{\frac{2}{n}}}.
\]
Then there is a constant $C=C_{n,S}>0$ such that for all $T,\eta>0$, $\delta\in(0,1]$, we have
\[
E_\eta(o,T)
\leq 
C\euler^{2\eta\sqrt {\delta (T\wedge R^2)}}
\frac{\tilde \Gamma\left(\sqrt{\delta (T\wedge R^2)}\right)^2}{(a\delta (T\wedge R^2))^{\frac{n}{2}}},
\]
where 
\[
\tilde  \Gamma^{n}(r)^2=\left[1\vee\frac{a^{\frac{n}{2}}}{m(o)}\right]^{\theta(r)}.
\]
\end{theorem}
\begin{proof}
Since $\bar \omega_T=2\eta\rho(o,\cdot)-2Th(\eta)$ satisfies the eikonal inequality by Lemma~\ref{lem:omegaeikonal}, the left-hand side is non-increasing in $T$. Hence, we can restrict to the case $T\leq R^2$. Our assumption implies the Faber-Krahn inequality with the same constants $a$ and $n$ for all  $r\in[0,R]$, i.e., for all $U\subset B(r)$,
\[
\lambda(U)\geq a \ m(U)^{-{\frac{2}{n}}}.
\]
Let $\delta\in(0,1]$ and $R'=\sqrt{\delta T}$. Clearly, $R'\leq \sqrt T\leq R$. For $R'\geq 288S$, apply
Lemma~\ref{lem:FKintheat} to $R'$ to get 
\[
\sum_{X}m \ p_T(o,\cdot)^2\euler^{2\eta(\rho(o,\cdot)-R')_+-2Th(\eta)}
\leq C_{n,S}'
\left[1\vee\frac{a^{\frac{n}{2}}}{m(o)}\right]^{\theta(R')}
\frac{1}{(a(T\wedge (R')^2))^{\frac{n}{2}}}
=C_{n,S}'
\frac{\tilde \Gamma(\sqrt{\delta T})^2}{(a\delta T)^{\frac{n}{2}}}
\]
with $C_{n,S}'=(1\vee S)^n2^{2n^2+27n+51}$. 
This implies the claim in the first case.
\\
If $R'=\sqrt{\delta T}<288S$, since $\theta(R')=(n+2)/n\geq1$, we get by Lemma~\ref{lem:kappabound}(i) 
\[
\sum_{X}m \ p_T(o,\cdot)^2\euler^{2\eta(\rho(o,\cdot)-R')_+-2Th(\eta)}
\leq 
\frac{1}{m(o)}
=(\delta T)^{\frac{n}{2}}\frac{a^{\frac{n}{2}}}{m(o)}
\frac{1}{(a\delta T)^{\frac{n}{2}}}
\leq (288S)^\frac{n}{2} 
\frac{\tilde \Gamma(\sqrt{\delta T})^2}{(a\delta T)^{\frac{n}{2}}}.
\]
In order to bound the exponential inside the sum on the left-hand side, use
\[2\eta( \rho(o,\cdot)-R')_+-2Th(\eta)\geq 2\eta\rho(o,\cdot)-2Th(\eta)-2\eta R'.
\]
The claim follows by setting $C_{n,S}=(288S)^\frac{n}{2}\vee C_{n,S}'$.
\end{proof}

\section{Relative Faber-Krahn and heat kernel  bounds}\label{section:FKhkb}

In this section, we use bounds on the weighted $\ell^2$-norms of the heat kernel from Section~\ref{section:weightedl2norm} to obtain Gaussian upper heat kernel estimates with the optimal Gaussian for graphs with an intrinsic metric. 
\\
For $x\in X$, $\eta, t>0$, recall the weighted $\ell^2$-norm of the heat kernel
\[
E_\eta(x,t)=\sum_{z\in X}m(z)
p_t(x,z)^2
\euler^{
2\eta\rho(x,z)-2th(\eta)},
\]
where 
\[
h(\eta)=S^{-2}\big(\cosh(\eta S)-1\big).
\]
The use of variants of  $E_\eta$ to bound heat kernels has been used in the past, e.g., in \cite{ChengLiYau-81}, and this technique has been refined by Grigor'yan \cite{Grigoryan-94}.
In contrast to earlier works, we combine it with Davies' trick to obtain new a priori heat kernel estimates involving the sharp Gaussian in terms of $E_\eta$ and appropriate $\eta>0$.

\begin{theorem}\label{thm:split}For all $x,y\in X$, $t, t_0>0$, we have 
\begin{align*}
p_{t}(x,y)\leq
\sqrt{ E_{\eta_0}(x,(t\wedge t_0)/2)
E_{\eta_0}(y,(t\wedge t_0)/2)}
 \euler^{-\Lambda(t-t\wedge t_0)
-\zeta(\rho(x,y),t)},
\end{align*}
where $\eta_0=\frac{1}{S}\arsinh\left(\frac{\rho(x,y) S}{{t}}\right)$.
\end{theorem}
\begin{proof}For any $\eta>0$, the semigroup property yields
\begin{multline*}
p_{t}(x,y)
=\sum_{z\in X}m(z)
p_{t/2}(x,z)
\euler^{
\eta\rho(x,z)-({t/2})h(\eta)}
p_{t/2}(y,z)
\euler^{\eta\rho(y,z)-({t/2})h(\eta)
}
 \euler^{
-\eta(\rho(x,z)+\rho(y,z))+th(\eta)}
\\
\leq  \euler^{
-\eta\rho(x,y)+th(\eta)}
\sum_{z\in X}m(z)
p_{t/2}(x,z)
\euler^{
\eta\rho(x,z)-({t/2})h(\eta)}
p_{t/2}(y,z)
\euler^{\eta\rho(y,z)-({t/2})h(\eta)
},
\end{multline*}
where we used the triangle inequality to obtain $
-\eta(\rho(x,z)+\rho(y,z))+th(\eta)\leq 
-\eta\rho(x,y)+th(\eta)$.
Further,
the H\"older inequality yields for the sum
\begin{multline*}
\sum_{z\in X}m(z)
p_{t/2}(x,z)
\euler^{
\eta\rho(x,z)-({t/2})h(\eta)}
p_{t/2}(y,z)
\euler^{\eta\rho(y,z)-({t/2})h(\eta)
}
\\
\leq \left(\sum_{z\in X}m(z)
p_{t/2}(x,z)^2
\euler^{
2\eta\rho(x,z)-th(\eta)}\right)^{\frac{1}{2}}
\left(\sum_{z\in X}m(z)p_{t/2}(y,z)^2
\euler^{2\eta\rho(y,z)-th(\eta)}\right)^{\frac{1}{2}}
=\sqrt{E_\eta(x,{t/2})E_\eta(y,{t/2})}.
\end{multline*}
Since $\bar\omega_t=2\eta\rho(z,\cdot)-2th(\eta)$ satisfies the eikonal inequality by Lemma~\ref{lem:omegaeikonal} for $z\in\{x,y\}$, it follows from the integral maximum principle that for $z\in\{x,y\}$ and all $\eta>0$, and $t,t_0\geq0$
\[
E_\eta(z,{t/2})\leq E_\eta(z,(t\wedge t_0)/2)\euler^{-\Lambda(t-t\wedge t_0)},
\]
where we used $t\geq t\wedge t_0$.
Next, we choose $$\eta=\eta_0=\frac{1}{S}\arsinh\left(\frac{\rho(x,y) S}{{t}}\right),$$ 
which is the minimizer of the function $\eta\mapsto -\eta\rho(x,y)+th(\eta)$.
Since $h(\eta_0)=\tfrac{1}{S^2}\big(\cosh(\eta_0 S)-1\big)$ and $\cosh(\arsinh(x))=\sqrt{x^2+1}$,  we get
\[
{t} h(\eta_0)
=\frac{{t}}{S^2}
\left[
\cosh
\left(
\arsinh
\left(
\frac{\rho(x,y) S}{{t}}
\right)
\right)-1
\right]
\\
=\frac{1}{S^2}\left(\sqrt{\rho(x,y)^2S^2+t^2} -{t}
\right)
\]
and hence 
\[
-\eta_0 \rho(x,y)+{t} h(\eta_0)
=-\frac{\rho(x,y)}{S}\arsinh\left(\frac{\rho(x,y) S}{{t}}\right)+\frac{1}{S^2}\left(\sqrt{\rho(x,y)^2S^2+t^2} -{t}
\right)
=
-\zeta\big(\rho(x,y),{t}\big).
\]
Plugging in this estimate and using the monotonicity of $E$ described above yields the claim.
\end{proof}

Next, we turn to the deriviation of Gaussian upper bounds for the heat kernel given Faber-Krahn inequalities in balls. 
In order to formulate the next theorem, recall
\[
\theta(n,r)=\frac{n+2}{n}\cdot 1\wedge \left(\frac{288S}{r}\right)^{\frac{1}{n+2}}.
\]

\begin{theorem}\label{thm:upperboundgeneral}
Let $x,y\in X$, $a_x,a_y$, $R_2\geq 2\cdot 288S$,  such that for $z\in\{x,y\}$ and all $U\subset B_z(\sqrt t\wedge R_2)$ 
\[
\lambda(U)\geq a_z m (U)^{-{\frac{2}{n_z}}}.
\] 
Then we have for all $t\geq 2\cdot (288S)^2$
\[
p_{{t}}(x,y)
\leq
\Gamma^n_{x}(\sqrt{\tau_\rho})\Gamma^n_{y}(\sqrt{\tau_\rho})
\left[\frac{1\vee \frac{{t}}{S^2}\arsinh^2\left(\frac{\rho(x,y) S}{{t}}\right)}{\sqrt{a_xa_y}({t}\wedge R_2^2)}\right]^{\frac{n_{xy}}{2}}
\euler^{-\Lambda({t}-{t}\wedge R_2^2)-\zeta(\rho(x,y),{t})},
\]
where 
\[
n_{xy}:= \tfrac{n_x+n_y}2,\qquad 
\Gamma^{n}_z(\sqrt{\tau_\rho})=
C_{n_z,S}
\left[1\vee\frac{a_z^{\frac{n_z}{2}}}{m(z)}\right]^\frac{\theta\left(n_z,\sqrt{\tau_\rho}\right)}{2},
\qquad
\tau_\rho=\frac{\tau}{2}\wedge \frac{S^2}{2\arsinh^2(\rho(x,y) S/\tau)},
\]
and $C_{n_z,S}>0$, $z\in\{x,y\}$.
\end{theorem}

\begin{proof} For simplicity abbreviate $\tau:=t\wedge R^2_2$ and $\rho:=\rho(x,y)$. 
Theorem~\ref{thm:split} yields with $\tau=t\wedge R_2^2$ and $t_0=R_2^2$
\[
p_{{t}}(x,y)
\leq \sqrt{E_{\eta_0}(x,\tau/2)E_{\eta_0}(y,\tau/2)}
\ \euler^{-\Lambda({t}-\tau)
-\zeta(\rho,t)},
\]
where $\eta_0={S}^{-1}\arsinh\left({\rho S}/{{t}}\right)$.
Next, we bound $E_{\eta_0}(z,\tau/2)$ for $z\in\{x,y\}$.
We have $\sqrt\tau=\sqrt t\wedge R_2\geq 288S$. Hence,
Theorem~\ref{thm:integratedheat} applied to $a=a_z$, $n=n_z$, $R=\sqrt \tau$, $T=\tau/2$, yields for all $z\in\{x,y\}$, $\delta\in(0,1]$, and $\eta>0$
the estimate
\[
E_{\eta_0}(z,\tau/2)
=E_{\eta_0}(z,(t\wedge R^2_2)/2)
\leq
C_{n_z,S}\euler^{2{\eta_0}\sqrt {\delta (T\wedge R^2)}}
\frac{\tilde \Gamma^{n_z}(\sqrt{\delta (T\wedge R^2)})^2}{(a_z\delta (T\wedge R^2))^{\frac{n_z}{2}}}
=C_{n_z,S}'\euler^{2{\eta_0}\sqrt {\delta \tau/2}}
\frac{\tilde \Gamma^{n_z}(\sqrt{\delta \tau/2})^2}{(a_z\delta \tau)^{\frac{n_z}{2}}}
\]
where we we set $C_{n_z,S}'=2^{n_z/2}C_{n_z,S}$ with $C_{n_z,S}$ given in Theorem~\ref{thm:integratedheat}.
Choose 
\[
\delta=1\wedge \frac{1}{\tau\eta_0^2}=1\wedge \frac{S^2}{\tau\arsinh^2(\rho S/t)}.
\]
The estimate $\arsinh(\rho S/t)\leq \arsinh(\rho S/\tau)$ yields
\[
\sqrt{\delta \tau/2}=\sqrt{\frac{\tau}2\wedge \frac{S^2}{2(\arsinh^2(\rho S/t))} }
\geq 
\sqrt{\frac{\tau}2\wedge \frac{S^2}{2(\arsinh^2(\rho S/\tau))} }
= \sqrt{\tau_\rho}.
\] 
Since $\euler^{2\eta_0\sqrt {\delta \tau/2}}\leq \euler^2\le 2^4$ and $\theta(n,r)$ is monotone decreasing in $r$, we get $\tilde \Gamma^{n_z}(\sqrt{\delta \tau/2})^2\leq \tilde \Gamma^{n_z}(\sqrt{\tau_\rho})^2$ and thus
\[
E_{\eta_0}(z,\tau/2)
\leq 2^4C_{n_z,S}'
\tilde \Gamma^{n}_z(\sqrt{\tau_\rho})^2\frac{\left[1\vee \frac{\tau\arsinh^2(\rho S/t)}{S^2}\right]^{\frac{n_z}{2}}}{
(a_z\tau)^{\frac{n_z}{2}}}
=\Gamma^{n_z}(\sqrt{\tau_\rho})^2\frac{\left[1\vee \frac{\tau\arsinh^2(\rho S/t)}{S^2}\right]^{\frac{n_z}{2}}}{
(a_z\tau)^{\frac{n_z}{2}}},
\]
where we put in the definition of $\Gamma$ the constant $C_{n_z,S}:=
\sqrt{2^4C_{n_z,S}'}$.
Plugging this estimate for $z\in \{x,y\}$ into the above bound on $p_t(x,y)$ and 
using $\tau/S^2\leq t/S^2$ in the parentheses yields the claim.
\end{proof}

\begin{remark}In contrast to the upper heat kernel bound in Theorem~\ref{thm:upperboundgeneral}, the polynomial correction term in \cite{KellerRose-22a} is given by 
\[
1\vee S^{-2}\left(\sqrt{t^2+\rho^2S^2}-t\right).
\]
A simple application of l'Hospital's rule shows that the function 
\[
x\longmapsto \frac{\arsinh^2\left(x\right)}{\sqrt{1+x^2}-1}
\]
is bounded by $2$. Choosing $x=\rho S/t$ shows that the our polynomial correction term is bounded above by the one obtained in \cite{KellerRose-22a}.
\end{remark}

Finally, we show that the relative Faber-Krahn inequality (FK) yields the Gaussian upper bounds (G) with correction function $\Psi$, which we recall to be given by
\[
\Psi_{xy}(\sqrt \tau)^2=\Phi_y^{n(\sqrt\tau)}(\sqrt{\tau_\rho})\Phi_y^{n(\sqrt\tau)}(\sqrt{\tau_\rho}),
\qquad 
\Phi_z^{n}(t)=\nu_{z,n}^{\theta\left(n,t\right)}, \qquad \nu_{z,n}=1\vee \Deg_z^{\frac{n}{2}},
\]
where $\tau=t\wedge R_2^2$, $\tau_\rho=({\tau}/{2})\wedge {S^2}/({2\arsinh^2(\rho(x,y) S/\tau)})^{-1}$.

\begin{theorem}\label{thm:gaussianboundclean}
Let $x,y\in X$ and $R_2\geq 2R_1\geq 2\cdot 288S$. If there is a function $a\colon \{x,y\}\to(0,1]$ such that 
$FK(R_1,R_2,a,n)$ holds in $\{x,y\}$, then there exists a function $C=C_{a_x,a_y,n_x,n_y,S}\geq 1$
such that for all $t\geq R_1^2$ we have
\[
p_t(x,y)\leq C\Psi_{xy}(\sqrt t\wedge R_2)
\frac{\left(1\vee \frac{t}{S^2}\arsinh^2\left(\frac{\rho(x,y) S}{t}\right)\right)^{\frac{n_{xy}(\sqrt t\wedge R_2)}{2}}}{\sqrt{m(B_x(\sqrt t\wedge R_2))m(B_y(\sqrt t\wedge R_2))}}
\euler^{-\Lambda(t-t\wedge R_2^2)-\zeta(\rho(x,y),t)}
\]
where $n_{xy}(\sqrt t\wedge R_2):= \tfrac{n_x(\sqrt t\wedge R_2)+n_y(\sqrt t\wedge R_2)}2$.
\end{theorem}

\begin{proof}Fix $t\geq R_1^2\geq (2\cdot 288S)^2$, let $\sqrt\tau=\sqrt t\wedge R_2\in[R_1,R_2]$, and abbreviate $\rho:=\rho(x,y)$, $n_z:=n_z(\sqrt\tau)$, $a_z:=a_z(\sqrt \tau) $. Apply Theorem~\ref{thm:upperboundgeneral} with  \[a_z'=a_z'(\sqrt\tau)=\frac{a_z}{\tau}m(B_z(\sqrt \tau))^{{\frac{2}{n_z}}}, \quad z\in\{x,y\},\] to obtain
for all $t\geq (288S)^2$
\begin{multline*}
p_{t}(x,y)
\leq 
\Gamma^{n}_{x}(\sqrt{\tau_\rho})\Gamma^{n}_{y}(\sqrt{\tau_\rho})
\left[\frac{1\vee \frac{t}{S^2}\arsinh^2\left(\frac{\rho S}{t}\right)}{\sqrt{a_x'a_y'}\tau}\right]^{\frac{n_{xy}}{2}}
\euler^{-\Lambda(t-\tau)-\zeta(\rho,t)}
\\
=
\frac{\Gamma^{n}_{x}(\sqrt{\tau_\rho})\Gamma^{n}_{y}(\sqrt{\tau_\rho})}{(a_xa_y)^{\frac{n_{xy}}{4}}}
\frac{\left(1\vee \frac{t}{S^2}\arsinh^2\left(\frac{\rho S}{t}\right)\right)^{\frac{n_{xy}}{2}}}{\sqrt{m(B_x(\sqrt \tau))m(B_y(\sqrt \tau))}}
\euler^{-\Lambda(t-\tau)-\zeta(\rho,t)},
\end{multline*}
where $\Gamma^n_z$ contains the term $a_z'$. 
We are left with estimating $\Gamma^{n}_{z}(\sqrt{\tau_\rho})$ from above, $z\in\{x,y\}$. We start with some preparations. By Lemma~\ref{lem:localreg}, $FK(R_1,R_2,a,n)$ in $\{x,y\}$ yields the local regularity property $L(R_1,R_2,C\nu,n)$ in $\{x,y\}$, where $C=
\left(4/a_z(R)\right)^{\frac{n_z(R)}{2}}$, $z\in\{x,y\}$. Abbreviate further  $\theta_z:=\theta\left(n_z(\sqrt \tau),\sqrt {\tau_\rho}\right)$, $z\in\{x,y\}$. By the definition of $\Gamma^{n_z}_z$, we obtain
for $t\geq R_1^2$ and $z\in\{x,y\}$
\[
\Gamma^{n}_{z}(\sqrt{\tau_\rho})=C_{n_z,S}a_z^{\frac{n_z\theta_z }{4}}\left[\frac{m(B_z(\sqrt \tau))}{m(z)(\sqrt \tau)^{n_z}}\right]^{\frac{\theta_z}{2}}
\leq C_{n_z,S}a_z^{\frac{n_z\theta_z }{4}}\left(\frac{4}{a_z}\right)^{{\frac{n_z\theta_z}{4}}}\nu_z(R)^{\frac{\theta_z}{2}}
\leq 
2^{n_z}C_{n_z,S}\nu_z(R)^{\frac{\theta_z}{2}},
\]
where we used that  $\theta\leq 1$ implies $ a_z^{\frac{n_z\theta_z }{4}}\left(\frac{4}{a_z}\right)^{{\frac{n_z\theta_z}{4}}}=4^{{\frac{n_z\theta_z}{4}}}\leq 2^{n_z}$ and $C_{n_z,S}$ given by Theorem~\ref{thm:upperboundgeneral}.
Therefore,
\[
\frac{\Gamma^{n}_{x}(\sqrt{\tau_\rho})\Gamma^{n}_{y}(\sqrt{\tau_\rho})}{(a_xa_y)^{\frac{n_{xy}}{4}}}
\leq 2^{n_x}C_{n_x,S}2^{n_y}C_{n_y,S}\frac{\Psi_{xy}(\sqrt \tau)}{(a_xa_y)^{\frac{n_{xy}}{4}}}.
\]
Denoting the function $C=2^{n_x}C_{n_x,S}2^{n_y}C_{n_y,S}(a_xa_y)^{-\frac{n_{xy}}{4}}$ yields the claim.
\end{proof}

The following corollary for the normalizing measure is an immediate consequence.

\begin{corollary}[Gaussian bounds normalizing measure]\label{cor:gaussianboundcleannormalized}
Let $m=\deg$, $R_2\geq 2R_1\geq 600$, $a\in(0,1]$, $n>0$. If
$FK(R_1,R_2,a,n)$ holds in $\{x,y\}$, then there exists $C=C_{a,n}>0$ such that $G(R_1,R_2,C,n)$ holds in $\{x,y\}$.
\end{corollary}
\begin{proof}
If $R_0\geq 2\cdot 288$, since $\Deg=1$ and $S=1$, we infer from Theorem~\ref{thm:gaussianboundclean} the bounds $G(R_1,R_2,\tilde C,n)$ in $\{x,y\}$ with $\tilde C=C_{a,n,1}>0$.
\end{proof}

\section{From heat kernel bounds, regularity and  doubling to Faber-Krahn}\label{section:gausstoFK}

In this paragraph, we prove  (G) \ \& \ (V)  \ $\Rightarrow$ \ (FK) for the normalizing measure and fixed dimension as well as  (G) \ \& \ (V) \ \& \ (L) \ $\Rightarrow$ \ (FK) for arbitrary measures and variable dimensions, both for large balls. At the end of this section, we will prove the theorems in the introduction.
\\

The general strategy follows \cite{Grigoryan-09}. However, since we assume (G) \ \& \ (V) only on large balls, the resulting property (FK) contains an unpleasant error depending on the smallest measure in the considered balls. While this is good enough in case of the normalizing measure, a new idea is needed in order to obtain non-trivial bounds in (FK) for arbitrary measures. This is resolved by making use of  (L) and the introduction of a variable dimension in (FK) which absorbs the error and becomes controllable on large scales.
\\

In order to formulate the results of this section, we introduce the following notation for on-diagonal heat kernel bounds in balls.
\begin{definition}[Diagonal upper bound]Let $0\leq r_1\leq r_2$, $\Psi\colon X\times [r_1,r_2]\to[1,\infty)$ and $B\subset X$. The on-diagonal bound $O(r_1,r_2,\Psi)$ holds in $B$ if for all $x\in B$ we have
\[
p_t(x,x)\leq \frac{\Psi_x(\sqrt t)}{m(B_x(\sqrt t))},\quad t\in[r_1^2,r_2^2].
\]
\end{definition}
\begin{remark}
Clearly, $G(R_1,R_2,\Psi,n)$ in $B_o(R_2)$ implies $O(R_1,R_2,\tilde\Psi)$ in $B_o(R_2)$ with $\tilde\Psi_x(\sqrt t)=\Psi_{xx}(\sqrt t)$.
\end{remark}

First, we derive lower bounds on Dirichlet eigenvalues of subsets of balls assuming (G) \& (V) on large balls. 
Therefore, we make use of the proposition below. 
 \begin{proposition}[Comparing balls, {\cite[Lemma~5.1]{KellerRose-22a}}]\label{lem:ballcomparison}
 	Let $r\ge 8S$, $ n>0$, $\alpha,\Phi \geq 1$  be constants, 
 	and assume $V (r/4,r,\alpha\Phi,n)$ in $B_{o}(r)$. Then there for all $x,y\in B_o(r)$, we have 
 	\[
 	m(B_x(r))
 	\leq
 	2^{18n}\alpha^9\Phi ^{9}
 	\ m(B_{y}(r)).
 	\]
 \end{proposition}

The following lemma gives a general lower bound for the first eigenvalue of a subset of a ball given heat kernel bounds and doubling on large scales. A variable dimension is already introduced here to deal with the correction terms $\Psi$ and $\Phi$ coming from (G) and (V), respectively. Our argument is an elaboration on  \cite{Grigoryan-09} and presented to get precise dependence on the constants for later use.
\begin{lemma}\label{lem:FKapriori}Let  $0\leq \hat r\leq r$, $n>0$, $\alpha,\Phi,\Psi\geq 1,$ $o\in X$, 
and assume $O(\hat r,r,\alpha\Psi)$
and $V(\hat r,r,\alpha\Phi,n)$ in $B_o(r)$.
Then, for all $U\in B_o(r)$ we have 
\[
\lambda(U)\geq \sup_{t\in [\hat r^2, r^2]}\frac{1}{t}\ln\left(\frac{1}{2^{18n}\alpha^{11}\Phi^{10}\Psi}\frac{m(B_o(r))}{m(U)}\left(\frac{\sqrt t}{r}\right)^{n}\right).
\]
\end{lemma}

\begin{proof}
Let $U\subset B_o(r)$, $p^U$ the Dirichlet heat kernel of $U$, $\lambda_i$, $i\in\{1,\ldots, \vert U\vert\}$ the Dirichlet eigenvalues of $U$ and denote $\lambda_1=\lambda$. Since $p^U\leq p$ and $O(\hat r,r,\alpha\Psi)$ in $B_o(r)$, we have for all $t\in[\hat r^2,r^2]$
\[
e^{-\lambda t}\leq \sum_{i=1}^n \euler^{-\lambda_i t}=\sum_U m\ p_t^U(\cdot,\cdot)\leq \sum_Um\ p_t(\cdot,\cdot)\le \alpha\Psi\sum_{x\in U}\frac{m(x)}{ m(B_x(\sqrt t))}.
\]
In order to bound the sum on the right-hand side from above, we employ the doubling property. We obtain for all $t\in [\hat r^2, r^2]$ by Proposition~\ref{lem:ballcomparison} and the volume doubling property for all $x\in U\subset B_o(r)$
\[
\frac{m(B_o(r))}{m(B_x(\sqrt t))}
\leq 2^{18n}\alpha^9\Phi ^{9}\frac{m(B_x(r))}{m(B_x(\sqrt t))}\leq \alpha\Phi 2^{18n}\alpha^9\Phi ^{9} \left(\frac{r}{\sqrt t}\right)^{n}=2^{18n}\alpha^{10}\Phi ^{10} \left(\frac{r}{\sqrt t}\right)^{n}.
\]
Hence, we obtain 
\[
\alpha\Psi\sum_{x\in U}\frac{m(x)}{ m(B_x(\sqrt t))}
\leq 
\frac{2^{18n}\alpha^{11}\Phi ^{10}\Psi}{m(B_o(r))} \ m(U)\left(\frac{r}{\sqrt t}\right)^{n}.
\]
Plugging in this estimate for the exponential of $\lambda(U)$ and rearranging yield the claim.
\end{proof}

The following lemma shows that doubling on large scales implies reverse doubling on large scales. We follow \cite{Grigoryan-09} and track the constants for late use.

\begin{lemma}[reverse doubling]\label{lem:reversedoubling}
$32S\leq 8\hat r\leq r\leq \diam X/2$, $o\in X$, constants $n>0$, $\alpha,\Phi\geq 1$, and assume 
$V(\hat r,r,\alpha\Phi,n) $ in $B_o(r)$.
Then, we have 
\[
m(B_o(r_2))\geq \frac12\left(\frac{r_2}{r_1}\right)^{\eta} m(B_o(r_1)),\quad 4\hat r\leq r_1\leq r_2\leq r/2,
\]
where
$
\eta=2^{-21n}\alpha^{-10}\Phi ^{-10}.
$
\end{lemma}
\begin{remark}Note that Proposition~\ref{lem:ballcomparison} and Lemma~\ref{lem:reversedoubling} are the only places in the proof of the main theorems where we use that the intrinsic metric is a path degree metric. Instead of assuming the intrinsic metric is a path degree metric, we could also assume that the reverse volume doubling condition holds for the measure.
\end{remark}
\begin{proof} We mimick the classical proof and show volume comparison for balls where the radii differ only by factor two and iterate the resulting inequality.
\\
Let $r_1\geq4 \hat r$. Since $r\leq \diam X/2$ and $X$ is connected, there exists $y\in X\setminus B_o(r)$. Since $(X,\rho)$ is geodesic, we can take a a path $x=x_0\sim\dots\sim x_n=y$ from $o$ to $y$ realizing $\rho(o,y)$ and $N\in\{0,\ldots,n\}$ the minimal $N$ with the property $\rho(o,x_N)\geq 3r_1/2$. Then clearly $\rho(o,x_N)\leq S+3r_1/2$. Together with the triangle inequality, this yields $B_{x_N}(\tfrac{r_1}2-S)\subset B_o(2r_1)\setminus B_o(r_1)$.
Hence, we obtain 
\[
m(B_o(2r_1))\geq m(B_o(r_1))+m(B_{x_N}(\tfrac{r_1}2-S)).
\]
Since $S\leq \hat r/4$, $\hat r\leq r_1/4$, and $r_1\leq r/2$, we have 
\[\hat r=\hat r+S-S\leq 5\hat r/4-S\leq 5r_1/16-S\leq r_1/2-S\leq 2r_1\leq r.\]
Further, since $r_1\geq 4S$, we get   $\frac{r_1}{r_1/2-S}=\frac{2r_1}{r_1-2S} \leq \frac{2r_1}{r_1/2}=4$. Therefore, since $r/2\geq r_1\geq 4\hat r\geq 32S$,  by applying Proposition~\ref{lem:ballcomparison} to $B_{o}(2r_1)$ we obtain 
\[
\frac{m(B_x(r_1))}{m(B_{x_N}(r_1/2-S))}\leq 2^{18n}\alpha^9\Phi ^{9}
\frac{m(B_{x_N}(r_1))}{m(B_{x_N}(r_1/2-S))}
\leq 
2^{18n}\alpha^9\Phi ^{9}\alpha\Phi \left(
\frac{r_1}{r_1/2-S}\right)^{n}
\\\leq 
2^{20n}\alpha^{10}\Phi ^{10} .
\]
Hence, denoting $\tilde \alpha:= \frac{2^{20n}\alpha^{10}\Phi ^{10}+1}{2^{20n}\alpha^{10}\Phi ^{10}}$, we get
\[
m(B_o(2r_1))\geq m(B_o(r_1))+m(B_{x_N}(\tfrac{r_1}2-S))\geq \tilde \alpha\
m(B_o(r_1)).
\]

Now let $\hat r\leq r_1\leq r_2\leq r$ and $k\in\NN_0$ maximal with the property
\[
2^k\leq \frac{r_2}{r_1}<2^{k+1}.
\]
Then we obtain
\[
m(B_o(r_2))\geq m(B_o(2^kr_1))\geq \tilde \alpha\ 
m(B_o(2^{k-1}r_1))\geq \ldots \geq 
\tilde \alpha^{k}m(B_o(r_1)),
\]
and, since $k\geq \log_2\tfrac{r_2}{2r_1}$,
\[
\tilde \alpha^k
\geq  
\tilde \alpha^{\log_2\tfrac{r_2}{2r_1}}
=\tilde \alpha^{-1}
\left(\frac{r_2}{r_1}\right)^{\log_2 \tilde \alpha
}.
\]
We have
\[
\tilde \alpha^{-1}
=\frac{2^{20n}\alpha^{10}\Phi^{10} }{2^{20n}\alpha^{10}\Phi^{10} +1}\geq \frac{2^{20n}\alpha^{10}\Phi^{10} }{2\cdot 2^{20n}\alpha^{10}\Phi^{10}}=\frac12.
\]
By the mean value theorem and $\ln2\leq 1$ we have $\log_2\frac{x+1}x\geq {1}/({x+1})$
and thus, since $\alpha\geq 1$,
\[
\log_2\tilde \alpha
\geq \frac{1}{2^{20n}\alpha^{10}\Phi ^{10}+1}\geq \frac{1}{2^{21n}\alpha^{10}\Phi ^{10}}=\eta.\qedhere
\]
\end{proof}

Next, we derive (FK) from (O) and (V). At the heart is the following lemma which adds an additional assumption to Grigor'yan's argument in \cite{Grigoryan-09}. Later, we will use the new assumption for two different purposes; a fixed dimension and a variable dimension. If the dimension is assumed to be fixed, the condition will depend on the worst vertex degree and inverted measure inside the considered ball, such that all estimates will depend on these notions. If variable dimensions are allowed, together with the local regularity condition, the new condition can be used to absorb bad vertex degrees on large scales.

\begin{lemma}\label{lem:choicegammanprime}
Let $32S\leq 8\hat r\leq r\leq \diam M/2$, $o\in X$, and  constants $A,\alpha,\gamma,\Phi ,\Psi \geq 1$, and $n'\geq n>0$ satisfying
\[
A\geq \exp(2^{19n}\euler\Phi^{10}\Psi),\qquad 
\gamma\left(\frac{r}{\hat r}\right)^{n'}
 \geq m(B_o(r))\|\tfrac1m\|_{B_o(r)},
 \qquad
 \gamma\left(\frac{Ar}{\hat r}\right)^{n'}
 \geq m(B_o(Ar))\|\tfrac1m\|_{B_o(Ar)}.
\]
If $O(\hat r,2Ar,\alpha\Psi)$ 
and $ V(\hat r,2Ar,\alpha\Phi, n)$ hold in $B_o(Ar)$, then
$FK(r,r,a,n')$  holds in $o$, where
\[
a=\left[\frac{1}{\gamma \alpha^{11}\ln A}\right]^{\frac{2}{n}}\frac{1}{A^2}.
\]
\end{lemma}

\begin{proof}We adapt the proof from \cite{Grigoryan-09} to our setting.
Set $\epsilon=  {2}/(\gamma\alpha^{11}\ln A)\in(0,\gamma^{-1}]$ and
let  $U\subset B_o(r)$ be arbitrary but fixed. 

First, consider the case 
\[
m(U)\leq 
\epsilon
\ m(B_o(r)).
\] 
Since  $A\geq 1$ 
we have $O(\hat r,2r,\alpha\Psi)$ 
and $ V(\hat r,2r,\alpha\Phi, n)$ in $B_o(r)$. By Lemma~\ref{lem:FKapriori}, we get a lower bound on $\lambda(U)$ involving a supremum over the interval $[\hat r^2,r^2]$. We choose $t\in[\hat r^2,r^2]$ and adapt the constants accordingly.

\eat{
Set $\epsilon=  {2}/(\gamma\alpha^{11}\ln A)\in(0,\gamma^{-1}]$. 

First, consider the case 
\[
m(U)\leq 
\epsilon
\ m(B_o(r))
\]
and 
}
Let
\[
 t=\left(\frac1\epsilon 
\frac{m(U)}{m(B_o(r))}\right)^{\frac{2}{{n'}}}r^2.
\]
We have $t\leq r^2$ and infer from the three bounds 
\[
\frac1{\epsilon\gamma}\geq 1,\qquad m(U)\geq \inf_{B_o(r)}m=\frac1{\|\tfrac1m\|_{B_o(r)}},
\qquad\text{and}\qquad \gamma\left(\frac{r}{\hat r}\right)^{n'}
\geq \|\tfrac1m\|_{B_o(r)}m(B_o(r)),
\]
the estimate
\[
t=\left(\frac1{\gamma\epsilon}
\frac{\gamma m(U)}{m(B_o(r))}\left(\frac{r}{\hat r}\right)^{n'}\right)^{\frac{2}{{n'}}} \hat r^2
\geq \left(
\frac{\gamma}{\|\tfrac1m\|_{B_o(r)}m(B_o(r))}\left(\frac{r}{\hat r}\right)^{n'}\right)^{\frac{2}{{n'}}} \hat r^2
\geq \hat r^2,
\]
such that we have $t\in[\hat r^2,r^2]$. The aformentioned lower bound on $\lambda(U)$ inferred from Lemma~\ref{lem:FKapriori} involves the term $(\sqrt t/r)^n$. Since $\sqrt t\leq r$, we can estimate $n$ by a larger number $n'$ and therefore decrease this term. 
Putting $t$ into the resulting lower bound, and using the assumption $\Psi\Phi^{10}\leq (\ln A)/(2^{19n}\euler)$ and $\epsilon<1$, we get
\begin{multline*}
\lambda(U)\geq \sup_{t\in [\hat r^2, r^2]}\frac{1}{t}\ln\left(
\frac{1}{2^{18n}\alpha^{11}\Phi^{10}\Psi}\frac{m(B_o(r))}{m(U)}\left(\frac{\sqrt t}{r}\right)^{n'}\right)
\geq 
\frac{\epsilon^{\frac{2}{{n'}}} }{r^2}
\left(
\frac{m(B_o(r))}{m(U)}\right)^{\frac{2}{{n'}}}\ln\left(
\frac{1}{2^{18n}\alpha^{11}\Phi^{10}\Psi}\frac1{\epsilon}\right)
\\
\geq 
\frac{\epsilon^{\frac{2}{{n}}} }{r^2}
\left(
\frac{m(B_o(r))}{m(U)}\right)^{\frac{2}{{n'}}}\ln\left(
\frac{1}{2^{18n}\alpha^{11}}\frac{2^{19n}\euler}{\ln A}\frac1{\epsilon}\right)
=
\frac{\epsilon^{\frac{2}{{n}}} }{r^2}
\left(
\frac{m(B_o(r))}{m(U)}\right)^{\frac{2}{{n'}}}\ln\left(
\euler\gamma\right)
\geq \frac{a}{r^2}
\left(
\frac{m(B_o(r))}{m(U)}\right)^{\frac{2}{{n'}}},
\end{multline*}
\eat{We need to choose $\epsilon\in(0,\gamma^{-1}]$ such that  the logarithm is larger or equal to one. 
Take $\epsilon=  {2}/({\alpha^{11}\ln A}\gamma)$ and use $\gamma\geq 1$ to obtain
\[
\lambda(U)\geq \frac{\epsilon^{\frac{2}{{n}}} }{r^2}
\left(
\frac{m(B_o(r))}{m(U)}\right)^{\frac{2}{{n'}}}\ln (\euler\gamma)
\geq \frac{\epsilon^{\frac{2}{{n}}} }{r^2}
\left(
\frac{m(B_o(r))}{m(U)}\right)^{\frac{2}{{n'}}},
\]
}
what yields the result for the first case.

Secondly, assume 
\[
m(U)>
\epsilon\ m(B_o(r)).
\]
By assumption, we have  $O(\hat r,2Ar,\alpha\Psi)$ 
 and $ V(\hat r,2Ar,\alpha\Phi, n)$ in $B_o(Ar)$. By Lemma~\ref{lem:reversedoubling}, we have the reverse volume doubling property for radii up to $Ar\geq r$.
Hence, we obtain \[
\frac{m(B_o(Ar))}{m(B_o(r))}\geq 
\frac12A^{\eta},
\]
where 
$\eta=2^{-21n}\alpha^{-10}\Phi ^{-10}
$. 
Since $\alpha\geq 1$, the assumption on $A$ implies $2^{1+2n}/\alpha\leq \ln A$ and hence the estimate $\exp(\euler/2^{2n}\alpha^{10})\geq 4/(\alpha^{11}\ln A) $.
Together with $A\geq 1$ and $\Phi^{10}\leq \Phi^{10}\Psi\leq {\ln A}/({2^{19n}\euler})$, we obtain
\[
A^\eta=\exp\left(\frac{\ln A}{2^{21n}\alpha^{10}\Phi ^{10}}\right)
\geq \exp\left(\frac{\ln A}{2^{21n}\alpha^{10}\frac{\ln A}{2^{19n}\euler}}\right)
=
\exp\left(\frac{\euler}{2^{2n}\alpha^{10}}\right)
\geq \frac{2}{\alpha^{11}\ln A}\geq \frac{2}{\epsilon}.
\]
Therefore, 
\[
m(U)
\leq 
\epsilon\
m(B_o(Ar)).
\]
Since our assumptions hold for radii up to $Ar$, we can argue as  in the first case replacing the radius $r$ by the radius $Ar$ and using 
\[
m(U)\geq \inf_{B_o(Ar)}m=\frac1{\|\tfrac1m\|_{B_o(Ar)}}\qquad \text{and}\qquad
\gamma\left(\frac{Ar}{\hat r}\right)^{n'}
 \geq m(B_o(Ar))\|\tfrac1m\|_{B_o(Ar)}
\]
instead of the corresponding inequality for $A=1$ to obtain
\[
\lambda(U)
\geq 
\frac{
\epsilon^{\frac{2}{{n}}}}{(Ar)^2}
\left(
\frac{m(B_o(Ar))}{m(U)}\right)^{\frac{2}{{n'}}} 
\geq
\frac{a}{r^2}
\left(
\frac{m(B_o(r))}{m(U)}\right)^{\frac{2}{{n'}}}.
\]
This is the claim.
\end{proof}

\begin{theorem}[Normalizing measure]\label{thm:normalizedclassicalFK}
Let $m=\deg$, $32\leq 8R_1\leq R_2\leq \diam M/2$, $o\in X$, $\alpha\geq 1$, $n>0$,
and \[
A=\exp(2^{19n}\euler),
\qquad 
a=\left[\frac{1}{\gamma \alpha^{11}\ln A}\right]^{\frac{2}{n}}\frac{1}{A^2}.
\]
If $O(R_1,2AR_2,\alpha)$ 
and $ V(2AR_2,\alpha, n)$ hold in $B_o(AR_2)$, 
then $FK(R_2,R_2,a',n)$ holds in $o$.
\end{theorem}

\begin{proof}Recall $S=1$ and set $\Psi=\Phi=1$, $n'=n$, $\hat r=R_1$, $r=R_2$. Since distance balls are finite by assumption, there exists $x^\ast\in B_o(R_2)$ such that 
\[
m(B_o(R_2))\|\tfrac1m\|_{B_o(R_2)}=\frac{m(B_o(R_2))}{m(x^\ast)}\leq 2^{18n}\alpha^9\frac{m(B_{x^\ast}(R_2))}{m(x^\ast)}
=2^{18n}\alpha^9\frac{m(B_{x^\ast}(R_2))}{m(B_{x^\ast}(1/2))}
\]
where we used 
Lemma~\ref{lem:ballcomparison} and that the measure of balls with radius one half equals the measure of its center. By applying $ V(1/2,2AR_2,\alpha, n)$ in $B_o(AR_2)$ to the right-hand side, we obtain with $\gamma=2^{19n}\alpha^{10}A^n(1\vee R_1)^n$
\[
m(B_o(R_2))\|\tfrac1m\|_{B_o(R_2)}
\leq 
2^{19n}\alpha^{10}R_2^n
\leq \gamma\left(\frac{R_2}{R_1}\right)^{n}.
\]
By the same argument, choosing $\hat r=R_1$ and $r=AR_2$ yields
\[
m(B_o(AR_2))\|\tfrac1m\|_{B_o(AR_2)}\leq 
 \gamma\left(\frac{AR_2}{R_1}\right)^{n}.
\]
Lemma~\ref{lem:choicegammanprime} implies $FK(R_2,R_2,a,n)$ in $o$.
\end{proof}
Theorem~\ref{thm:normalizedclassicalFK} can be easily be  modified to include graphs with uniformly bounded vertex degree and $\mu$ uniformly bounded in $X$. The constant in (FK) would depend on the maximal vertex degree and become either zero or infinity if the vertex degree becomes unbounded on large scales. To overcome this obstacle, we allow for a variable dimension in (FK). It turns out that this dimension stays bounded if the maximal vertex degree in distance balls does not grow too fast depending on the radius. A main ingredient is the following observation.

\begin{theorem}[Choosing dimension]\label{thm:fixeddimension}
Let $32S\leq 8\hat r\leq r\leq \diam M/2$, $o\in X$, constants $\gamma,\alpha,\Phi ,\Psi \geq 1$, $n>0$, 
\[
n'\geq n\vee \ln\left(1\vee \gamma^{-1}\|\tfrac1m\|_{B_o(Ar)}m(B_o(Ar))\right)^{\frac{1}{\ln \left({r}/{\hat r}\right)}},
\] 
and
\[
A=\exp(2^{19n}\euler\Phi^{10}\Psi),
\qquad a=\left[\frac{1}{\gamma \alpha^{11}\ln A}\right]^{\frac{2}{n}}\frac{1}{A^2}.
\]
If $O(\hat r,2Ar,\alpha\Psi)$ 
and $ V(\hat r,2Ar,\alpha\Phi, n)$ hold in $B_o(Ar)$, then $FK(r,r,a,n')$ holds in $o$.
\end{theorem}

\begin{proof} Since $A\geq 1$, we have
\[
\left(1\vee \gamma^{-1}\|\tfrac1m\|_{B_o(Ar)}m(B_o(Ar))\right)^{\frac{1}{\ln \left({r}/{\hat r}\right)}}
\geq \left(1\vee \gamma^{-1}\|\tfrac1m\|_{B_o(r)}m(B_o(r))\right)^{\frac{1}{\ln \left({r}/{\hat r}\right)}},
\]
and 
\[
\left(1\vee \gamma^{-1}\|\tfrac1m\|_{B_o(Ar)}m(B_o(Ar))\right)^{\frac{1}{\ln \left({r}/{\hat r}\right)}}
\geq \left(1\vee \gamma^{-1}\|\tfrac1m\|_{B_o(Ar)}m(B_o(Ar))\right)^{\frac{1}{\ln \left({Ar}/{\hat r}\right)}},
\]
such that 
\[
\gamma\left(\frac{r}{\hat r}\right)^{n'}
 \geq m(B_o(r))\|\tfrac1m\|_{B_o(r)},
 \qquad
 \gamma\left(\frac{Ar}{\hat r}\right)^{n'}
 \geq m(B_o(Ar))\|\tfrac1m\|_{B_o(Ar)}.
\]Hence, Lemma~\ref{lem:choicegammanprime} with $\gamma\geq 1$ and $A= A'$ delivers $FK(r,r,a',n')$ in $o$ as desired.
\end{proof}

Next, we use Theorem~\ref{thm:fixeddimension} to obtain variable dimensions depending on the radius of the considered ball which scale according to the correction terms $\Psi$ and $\Phi$ appearing in (G) and (V).
\\
For $0\leq r$ and $o\in X$ recall $Q(r)=Q_o( (\ln r)/8,r)=B_o(r)\times [(\ln r)/8,r]$.
, 
and 
\[ 
\iota(r)=
\begin{cases}
\frac{1}{4} &\colon r\in[8R_1, \euler^{1\vee 8R_1}),\\
\frac{1}{2\ln(r/(\ln r)^{p(r)})}&\colon r\geq \euler^{1\vee 8R_1},
\end{cases}
\qquad 
p(r)=
\|n\|_{Q_o(r)}+3.
\]

\begin{remark}The parameter $p$ is chosen in order to obtain the desired doubling dimension of the variable dimension $n'$ in the limit in Theorems~\ref{thm:normalizedmain} and \ref{thm:countingmainglobal}.
\end{remark}
\begin{theorem}\label{thm:FKvariable}
Let $32S\leq 8R_1\leq R_2\leq \diam M/2$, $\gamma\colon [R_1,R_2]\to[1,\infty)$, $\alpha\colon X\times [R_1,R_2]\to[1,\infty)$, $o\in X$, 
\[
n'(r)\geq n(r)\vee \ln\left(1\vee \gamma(r)^{-1}\|\tfrac1m\|_{B_o(Ar)}m(B_o(Ar))\right)^{\frac{1}{\ln \left(8r/r'\right)}},
\] 
\[
r'=
\begin{cases}
{r} &\colon r\in[8R_1, \euler^{1\vee 8R_1}),\\
(\ln r)^{p(r)}&\colon r\geq \euler^{1\vee 8R_1},
\end{cases}
\] 
and
\[
A(r)=\exp(2^{19n(r)}\euler\|\Phi^{10}\|_{Q(r)}\|\Psi\|_{Q(r)}),
\qquad a(r)=\left[\frac{1}{\gamma(r) \|C\|_{Q(r)}^{11}\ln A(r)}\right]^{\frac{2}{n(r)}}\frac{1}{A(r)^2}.
\]
If $O(R_1,R_2,C\Psi)$ and $V(R_1,R_2,C\Psi,n)$ hold in $B_o(R_2)$, then
 for all $r\in [8R_1,R_2]$ with $2A(r)r\leq R_2$ we have $FK(r,r,a,n')$ in $o$.
\end{theorem}

\begin{proof}Let $r\in [8R_1,R_2]$ with $2A(r)r\leq R_2$. Set  
\[
r'=
\begin{cases}
{r} &\colon r\in[8R_1, \euler^{1\vee 8R_1}),\\
(\ln r)^{p(r)}&\colon r\geq \euler^{1\vee 8R_1}.
\end{cases}
\] 
Since $A(r)\geq 1$ we have $r\leq 2A(r)r\leq R_2$. In the following we show $R_1\leq r'/8 $ and $r'\leq r$. Then, we apply Theorem~\ref{thm:fixeddimension} to $\hat r=r'/8$ and $r$ with the suprema of the control terms in $Q(r)\supset B_o(r)\times [r'/8,r]$ to obtain the claim.
\\
First, let $r\in[8R_1, \exp(1\vee 8R_1))$. Then $r'=r$ and hence $R_1\leq r/8=r'/8$. 
\\
Secondly, assume $r\geq \exp(1\vee 8R_1)$. The condition implies $r\geq \euler$ and hence $\ln r\geq 1$. Since $p(r)\geq  1$, we have $(\ln r)^{p(r)}\geq \ln r $. Thus, since $r\geq \euler^{8R_1}$, we get $$r'/8=(\ln r)^{p(r)}/8\geq \frac18\ln r \geq R_1.$$ In particular, $r'\geq 8R_1\geq 32S$. Moreover, since $1/p(r)\leq 1$ and  $\ln x\leq x^{1/p}$, $p\geq 1$, $x>0$, we obtain
\[
r'=(\ln r)^{p(r)}\leq r.
\]
Thus, applying Theorem~\ref{thm:fixeddimension} to $\hat r=r'/8$ and $r$ yields the claim for all $n'$ satisfying
\[
n'\geq n(r)\vee \ln\left(1\vee \gamma(r)^{-1}\|\tfrac1m\|_{B_o(Ar)}m(B_o(Ar))\right)^{\frac{1}{\ln \left({8r}/{r'}\right)}}.
\] 

We are left with the estimation of the exponent inside the logarithm. In the case where $r\in[8R_1, \exp(1\vee 8R_1))$, we obtain
\[
\frac{1}{\ln \left({8r}/{r'}\right)}=\frac{1}{\ln 8}\le \frac{1}{2}=2\iota(r).
\]
In the second case where $r\geq \exp(1\vee 8R_1)$, we get 
\[
\frac{1}{\ln \left({8r}/{r'}\right)}
=
\frac{1}{\ln \left({8r}/{(\ln r)^{p(r)}}\right)}
\leq \frac{1}{\ln \left({r}/{(\ln r)^{p(r)}}\right)}=2\iota(r).
\]
This yields the claim.
\end{proof}

In order to prove Theorem~\ref{thm:FKlocalregclean}, we apply the above general results to the error terms appearing 
in the implication (FK) \ $\Rightarrow$ (L)\ \& \ (V)\ \& \ (G) in Lemma~\ref{lem:localreg} and Theorems~\ref{thm:asympdoubling} and \ref{thm:gaussianboundclean}. 
Therefore, we need several preparations involving the correction terms of the implication (FK) \ $\Rightarrow$ (L)\ \& \ (V)\ \& \ (G). We collect the resulting correction terms from the properties (L)\ \& \ (V)\ \& \ (G) first. Second, we bound the correction terms uniformly in space-time cylinders. Third, we use these uniform estimates to put them into the general context.
\\

In what follows, we fix $o\in X$, $0\leq 8R_1\leq r'\leq r\leq R_2$, $n\colon X\times[R_1,R_2]\to (0,\infty)$, and abbreviate $Q:=Q_o(r',r)=B_o(r)\times[r',r]$.
Assume  $L(R_1,R_2,C\nu)$ in $B_o(R_2)$ with 
$
C\colon X\times[0,\infty)\to[1,\infty)$ and recall $\nu_{z,n(R)}=
[1\vee\Deg_z]^{\frac{n_z(R)}{2}}$, $z\in B_o(R_2)$, $R\in[R_1,R_2]$.
We clearly obtain $L(r',r,\|C\|_Q\|\nu\|_{Q},\|n\|_Q)$ in $B_o(r)$, where
\[
\|\nu\|_{Q}\leq 1\vee \|\Deg\|_{B_o(r)}^{\frac{\|n\|_{Q}}{2}}.
\]
Further, let  $V(R_1,R_2,C\Phi,n) $ hold in $B_o(R_2)$ with 
\[ \Phi_z^{n(R)}(r)=\nu_{z,n(R)}^{\theta(n_z(R), r)},
\qquad
\theta(n,r)=\frac{n+2}{n}\cdot 1\wedge \left(\frac{288S}{r}\right)^{\frac{1}{n+2}}.
\]
This yields $V(r',r,\|C\|_{Q}\|\Phi\|_{Q},\|n\|_{Q})$ in $B_o(r)$. In order to bound $\|\Phi\|_Q$, observe that for all $R\in[r',r]$ we have  $n_x(R)/2+1\leq \|n\|_Q/2+1$, such that for all $R\in[r',r]$
\[
\frac{n_x(R)}{2}\theta(n_x(R),r)=\left(\frac{n_x(R)}{2}+1\right)\cdot 1\wedge \left(\frac{288S}{r}\right)^{\frac{1}{n_x(R)+2}}
\qquad\leq 
\frac{\|n\|_{Q}}{2}\theta(\|n\|_Q,r'/2).
\]
With this at hand, we obtain by the definition of the function  $\Phi$
\[
\|\Phi\|_{Q}
=
\sup_{(x,s)\in Q,t\in[r', s]}\nu_{x,n(s)}^{\theta(n_x(s),t)}
\leq
1\vee \|\Deg\|_{B_o(r)}^{\frac{\|n\|_{Q}}{2}\theta(\|n\|_Q,r'/2)}.
\]
Next, we turn to the correction term of the on-diagonal heat kernel bounds.
If we have $O(R_1,R_2,C\tilde\Psi,n)$ in $B_o(R_2)$ with
\[
\tilde \Psi_x(\sqrt\tau)^2=\Psi_{xx}(\sqrt \tau)^2=\nu_{x,n_x(\sqrt\tau)}^{\theta(n_x(\sqrt \tau),\sqrt {\tau_0})},\qquad \tau\in[R_1^2,R_2^2],
\]
where $\tau_0=\tau_{\rho(x,x)}$, then we have $O(r',r,\|C\|_{Q}^2\|\tilde \Psi\|_{Q})$ in $B_o(r)$.
In order to estimate the  $\tilde \Psi$-term, observe that   we have $\theta\left(n_z(\sqrt \tau),\sqrt {\tau_0}\right)=\theta\left(n_z(\sqrt \tau),\sqrt{\tau/2}\right)$, such that
$$\|\Psi\|_{Q}=\sup_{(x,s)\in Q,t\in[r', s]}\nu_{x,n(s)}^{\theta\left(n_x(s),t/2\right)}\leq 1\vee\|\Deg\|_{B_o(r)}^{\frac{\|n\|_{Q}}2{\theta(\|n\|_Q,r'/2)}}.$$

With these preparations we are now ready to state the next main theorem. 
In order to formulate it, let $0\leq 8R_1\leq r$, recall our notation for the space-time cylinders $Q(r)=Q_o( (\ln r)/8,r)=B_o(r)\times [(\ln r)/8,r]$ as well as $Q^A(r)=Q_o( (\ln r)/8,A(r)r)$ for some function $A\geq1$.
Our main theorem regarding the implication (G) \& (V) \& \ (L) \ $\Rightarrow$ \ (FK) including variable dimensions is as follows. 
\begin{theorem}\label{thm:FKlocalregclean}
Let $32S\leq 8R_1\leq R_2\leq \diam M/2$, $o\in X$, $C\colon X\times[0,\infty)\to [1,\infty)$, 
and 
\[
n'(r)= n(r)\vee \ln \left[[(1\vee A(r)^2r^2)\cdot 1\vee \|\Deg\|_{B_o(A(r)r)}]^{\|n\|_{Q^A(r)}(\iota(r)+10\theta(\|n\|_{Q^A(r)},r'/2)}\right],
\] 
\[
r'=
\begin{cases}
{r} &\colon r\in[8R_1, \euler^{1\vee 8R_1}),\\
(\ln r)^{p(r)}&\colon r\geq \euler^{1\vee 8R_1},
\end{cases}
\qquad 
p(r)=
\|n\|_{Q(r)}+3,
\] 
and
\[
A(r)=\exp\left(2^{19n(r)}\euler\|C\|_{Q(r)}^{11}\cdot 1\vee\|\Deg\|_{B_o(r)}^{11\frac{\|n\|_{Q(r)}}2{\theta(\|n\|_{Q(r)},r'/2)}}\right),\]
\[ a(r)=\left[\frac{1}{2^{18\|n\|_{Q^A(r)}}\|C\|_{Q^A(r)}^{13}\ln A(r)}\right]^{\frac{2}{n(r)}}\frac{1}{A(r)^2}.
\]
If $O(R_1,R_2,C\Psi)$,
$V(R_1,R_2,C\Phi,n) $, and $L(R_1,R_2,C\nu,n)$ hold in $B_o(R_2)$, then for all $r\in[8R_1,R_2]$ with $2A(r)r\leq R_2$ we have  $FK(r,r,a,n')$ in $o$.
\end{theorem}

\begin{proof}
Let $r\in[8R_1,R_2]$ with $2A(r)r\leq R_2$.
\eat{ and 
\[
r'=
\begin{cases}
{r} &\colon r\in[8R_1, \euler^{1\vee 8R_1}),\\
(\ln r)^{p(r)}&\colon r\geq \euler^{1\vee 8R_1}.
\end{cases}
\] 
}
In order to simplify the following computations we abbreviate for a set $Q$ $$n_Q:=\|n\|_Q,\quad C_Q:=\|C\|_Q, \quad \Psi_Q:=\|\Psi\|_Q,\quad \Phi_Q:=\|\Phi\|_Q. $$ We have
\[
A(r) \geq \exp(2^{19n(r)}\euler\Phi^{10}_{Q_o(r'/8,r)}\Psi_{Q_o(r'/8,r)})=:\tilde A(r).
\]
Hence, from Theorem~\ref{thm:FKvariable} we infer for any function $\gamma\geq 1$ the Faber-Krahn inequality $FK(r,r, \tilde a, \tilde n)$ with
\[
\tilde a(r)=\left[\frac{1}{\gamma(r) \|C\|_{Q_o(r'/8,r)}^{11}\ln \tilde A(r)}\right]^{\frac{2}{n(r)}}\frac{1}{\tilde A(r)^2},
\]
and $\tilde n(r)\geq 
 n_{Q_o(r'/8,r)}\vee \ln  T$, where
\[
\qquad T:=\left(1\vee \gamma(r)^{-1}\|\tfrac1m\|_{B_o(Ar)}m(B_o(Ar))\right)^{2\iota(r)}.
\]
We estimate $T$ in order to choose $\gamma$. Since distance balls are finite, for any $r\geq 0$ there exists $x^\ast\in B_o(r)$ with the property $\|\tfrac1m\|_{B_o(r)}m(B_o(r))={m(x^\ast)}^{-1}m(B_o(r))$. Hence, by $r\geq 4R_1$ and Proposition~\ref{lem:ballcomparison},
\[
\|\tfrac1m\|_{B_o(Ar)}m(B_o(Ar))={m(x^\ast)}^{-1}m(B_o(Ar))\leq 
2^{18n_{Q_o(r'/8,Ar)}}C_{Q_o(r'/8,Ar)}\Phi_{Q_o(r'/8,Ar)}^9\frac{m(B_{x^\ast}(r))}{m(x^\ast)}.
\]
Abbreviate $\varphi:=1\vee \|\Deg\|_{B_o(Ar)}$ and $\theta_Q:=\theta(n_{Q_o(r'/8,Ar)},r'/2)$. By the discussion before the theorem, $
\Phi_{Q_o(r'/8,Ar)}
\leq \varphi^{\frac{n_{{Q_o(r'/8,Ar)}}}{2}\theta_Q}.
$
Furthermore, since $L(r',r,C\nu,n)$ in holds $B_o(Ar)$, the same discussion yields the estimate $\tfrac{m(B_{x^\ast}(r))}{m(x^\ast)}\leq C_{Q_o(r'/8,Ar)}r^{n_{Q_o(r'/8,Ar)}}\varphi^{\frac{n_{Q_o(r'/8,Ar)}}{2}}$. Therefore, if $\gamma(r)= 2^{18n_{Q_o(r'/8,Ar)}}C_{Q_o(r'/8,Ar)}^2$, we get
\begin{multline*}
\|\tfrac1m\|_{B_o(r)}m(B_o(r))
\leq 
2^{18n_{Q_o(r'/8,Ar)}}C_{Q_o(r'/8,Ar)}\varphi^{\frac{9n_{{Q_o(r'/8,Ar)}}}{2}\theta_Q}C_{Q_o(r'/8,Ar)}r^{n_{Q_o(r'/8,Ar)}}\varphi^{\frac{n_{Q_o(r'/8,Ar)}}{2}}
\\
\leq 
\gamma(r)r^{n_{Q_o(r'/8,Ar)}}\varphi^{\frac{n_{Q_o(r'/8,Ar)}}{2}(1+10\theta_Q)}.
\end{multline*}
Hence, recalling $\iota(r)=1/2\ln(8r/r')$, we get
\[
T=\left(\gamma(r)^{-1}\|\tfrac1m\|_{B_o(Ar)}m(B_o(Ar))\right)^{2\iota}
\leq (Ar)^{2n_{Q_o(r'/8,Ar)}\iota}\varphi^{n_{Q_o(r'/8,Ar)}\iota(1+10\theta_Q)}
\leq [(1\vee (Ar)^2)\varphi]^{n_{Q_o(r'/8,Ar)}(\iota+10\theta_Q)}.
\]
\eat{
Next, we bound $T_2$, which can be done analogously as in for $T_1$ if we replace $r$ by $A(r)r$ and keep $r'$ in the definitions. Hence, choosing $\gamma(r)=2^{18n_{Q(r'/8,A(r)r)}C_{Q(r'/8,A(r)r)}}$ leads to the estimate
\begin{multline*}
T_2=\left(\gamma(r)^{-1}\|\tfrac1m\|_{B_o(A(r)r)}m(B_o(A(r)r))\right)^{\frac{1}{\ln(8A(r)r/r')}}
\leq (A(r)r)^{\frac{n_{Q(r'/8,A(r)r)}}{\ln(8A(r)r/r')}}\varphi^{\frac{n_{Q(r'/8,A(r)r)}}{2\ln(8A(r)r/r')}(1+10\theta_{Q(r'/8,A(r)r)})}
\\
\leq [(1\vee (A(r)r)^2)\varphi]^{n_{Q(r'/8,A(r)r)}(\frac{1}{\ln(8A(r)r/r')}+11\theta_{Q(r'/8,A(r)r)})}.
\end{multline*}
}
\eat{By the discussion before the theorem, we also have \[
\Phi_Q
\leq \varphi^{\frac{n_{Q}}{2}\theta_Q}.
\]
Hence, since $\iota\geq 0$ and $\varphi\geq 1$, 
\[
T_2
=
\Psi_{Q}^2 \Phi_{Q}^{20}
\leq \varphi^{n_{Q}\theta_Q+10n_Q\theta_Q}
=\varphi^{11n_Q\theta_Q}
\leq  [(1\vee r)\varphi]^{n_Q(\iota+11\theta_Q)}.
\]
}
Putting everything together yields
\[
n(r)\vee \ln T
\leq n(r)\vee \ln \left[[(1\vee (Ar)^2)\varphi]^{n_{Q(r'/8,A(r)r)}(\iota+10\theta_{Q(r'/8,A(r)r)})}\right]
=n'(r),
\]
such that the claim follows.
\end{proof}

The remainder of this section is devoted to the proof of the theorems presented in the introduction. First, we prove Theorem~\ref{thm:normalizedmain} concerning the normalizing measure and fixed dimension, involving the uniform assumption on small balls. Then, we prove Theorem~\ref{thm:generalmainglobal} and derive Theorem~\ref{thm:countingmainglobal} as special case.

\begin{proof}[Proof of Theorem~\ref{thm:normalizedmain}]By Corollaries~\ref{cor:localregnormalized},\ref{cor:asympdoublingnormalized} and \ref{cor:gaussianboundcleannormalized}, if $FK(R,a,n)$ holds in $X$ we obtain a positive constant $C=C_{a,n}>0$ such that $V(C,n)$, and $G(R,C,n)$ hold in $X$. Conversely, if $V(C,n)$, and $G(R,C,n)$ are satisfied in $X$, then Theorem~\ref{thm:normalizedclassicalFK} yields a constant $a'=a'_{C,n,R_1}>0$ such that $FK(R,a',n)$ holds in $X$.
\end{proof}

\begin{proof}[Proof of Theorem~\ref{thm:generalmainglobal}]
\begin{enumerate}[(i)]
\item If $R\geq 288S$ and $FK(R,a,n)$ holds in $X$ for functions $a>0$, $n\geq 1$, Lemma~\ref{lem:localreg}, Theorem~\ref{thm:asympdoubling} and Theorem~\ref{thm:gaussianboundclean} yield the existence of a function $C=C_{a,n,S}>0$ such that $L(R,C\nu,n)$, $V(C\Phi,n)$ and $G(R,C\Psi,n)$ holds in $X$. 
\item Set $\vartheta(r)=\theta(\|n\|_{Q(r)}, (\ln r)^{p(r)}/16)$. Assume $R\geq R'$ and $L(R,C\nu,n)$, $V(C\Phi,n)$ and $G(R,C\Psi,n)$ hold for functions $C,n\geq1$ in $X$. By possibly increasing $n'(r)$ further by the requirement $$n'(r)\geq 1\vee\|\Deg\|_{B_o(r)}^{11\frac{\|n\|_{Q(r)}}2{\theta(\|n\|_{Q(r)},r'/2)}}\geq \Phi,\Psi,$$ Theorem~\ref{thm:FKlocalregclean} yields the existence of a function $a'=a'_{C,n'}>0$ such that $FK(8R,a',n')$ in $X$.\qedhere
\end{enumerate}
\end{proof}

\begin{proof}[Proof of Theorem~\ref{thm:countingmainglobal}]
This is Theorem~\ref{thm:generalmainglobal} for the case $m=1$.
\end{proof}

\noindent\textbf{Disclosure statement:}
The author has no competing interests to declare that are relevant to the content of this article.
\\

\noindent\textbf{Data availability:}
There is no data associated to this article.
{\scriptsize
\bibliographystyle{alpha}

}
\end{document}